\documentclass[12pt]{amsart}
\usepackage{amssymb,amsmath,amsthm}
 \usepackage[dvipsnames]{xcolor}

\usepackage[charter,cal=cmcal]{mathdesign}
\DeclareSymbolFont{usualmathcal}{OMS}{cmsy}{m}{n}
\DeclareSymbolFontAlphabet{\mathcal}{usualmathcal}

\def\e{{\rm e}}
\def\cic{\boldsymbol}
\def\eps{\varepsilon}

\def\d{{\rm d}}
\def\dist{{\rm dist}}

\def\R {\mathbb{R}}

\def\F {{\mathcal F}}

\def\T{{\mathsf{T}}}

\def\M {{\mathsf M}}

\def\Tt {{\mathcal T}}
\def \l {\langle}
\def \r {\rangle}

\def \and{\qquad\text{and}\qquad}

\newcommand{\supp}{\mathrm{supp}\,}

\newtheorem{proposition}{Proposition}
\newtheorem{theorem}{Theorem}
\newtheorem*{theorem*}{Theorem}
\newtheorem*{proposition*}{Proposition}

\newtheorem{lemma}[proposition]{Lemma}
\theoremstyle{definition}
\newtheorem{definition}[proposition]{Definition}
\newtheorem{remark}[proposition]{Remark}

\numberwithin{proposition}{section}
\numberwithin{equation}{section}

\title[Modulation invariant Carleson embedding]{A modulation invariant Carleson embedding theorem \\ outside local $L^2$}

\author{Francesco Di Plinio}
\address{\noindent Brown University Mathematics
Department, Box 1917,
Providence, RI 02912, USA}
\email{fradipli@math.brown.edu \textrm{(F.\ Di Plinio)} }

\author{Yumeng Ou}
\address{\noindent Brown University Mathematics
Department, Box 1917,
Providence, RI 02912, USA}
\email{yumeng\_ou@brown.edu \textrm{(Y.\ Ou)} }
\thanks{FDP was partially
supported by the National Science Foundation under the grant
   NSF-DMS-1500449. YO was partially supported by   the National Science Foundation under the grant
  NSF-DMS 0901139}

\begin{document}  
 \begin{abstract}  The article \cite{DoThiele15} develops a theory of  Carleson embeddings in   outer $L^p$ spaces for the wave packet transform
$$ F_\phi(f)(u,t,\eta)= \int f(x) \e^{i\eta(u-x)} \phi\left( \frac{u-x}{t}\right) \, \frac {\d x}{t}, \qquad (u,t,\eta) \in \R\times (0,\infty)\times \R$$
of functions $f \in L^p(\R)$, in the  $2\leq p\leq \infty$ range referred to as local $L^2$. In this article, we formulate  a   suitable extension of this theory to exponents $1<p<2$,  answering a question posed in \cite{DoThiele15}. The proof of our main embedding theorem  involves a refined  multi-frequency Calder\'on-Zygmund decomposition in the vein of \cite{DPTh2014,NOT}. We apply our embedding theorem  to recover the full known range of $L^p$ estimates for the bilinear Hilbert transforms \cite{LT2} without reducing to discrete model sums or appealing to generalized restricted weak-type interpolation. \end{abstract}
\maketitle

 \section{Introduction} 

  We are concerned with   the \emph{continuous wave packet transform} of   $f:\R \to \mathbb C$
  \begin{equation} 
\label{defFphi}
  F_\phi(f)(u,t,\eta)=  f* \phi_{t,\eta} (u), \qquad (u,t,\eta) \in \mathcal Z:= \R \times (0,\infty)\times \R\end{equation}
  where, for $t>0$ and $\eta\in \R$,
  $$
  \phi_{t,\eta}= \mathrm{Mod}_{\eta}\mathrm{Dil}_t^1 \phi, \qquad 
  \phi_{t,\eta}(x):= \e^{i\eta x}\textstyle\frac{1}{t}  \phi\left(\textstyle \frac{x}{t}\right)
  $$
is the \emph{wave packet} at scale $t$ and frequency $\eta$, and $\phi$ is a  nondegenerate real valued Schwartz function with compact frequency support.   The coefficients $ F_\phi(f)$, whose arguments  parametrize  the group of  symmetries of the class of modulation invariant singular integrals, provide an efficient description of the action of operators from this class. However, due to the overdetermination of the family $\{\phi_{t,\eta}: t\in (0,\infty),\, \eta \in \R\}$,   an effective control of the coefficient norms of $F_\phi(f)$ in terms of the norms of $f$   cannot be obtained via immediate orthogonality considerations, unlike the case of the \emph{wavelet transform} of $f$ where the modulation parameter  $\eta$ is fixed.

 A fundamental modulation invariant singular integral is 
the  (family of) trilinear form(s)  on  Schwartz functions  
 \[
\Lambda_{\vec\beta} (f_1,f_2,f_3)  = \int_\R \mathrm{p.v.} \int_{\R} f_1(x-\beta_ 1 t) f_2(x-\beta_2 t) f_3 (x-\beta_3 t) \, \frac{\d t}{t}\; \d x \qquad 
\] parametrized by unit vectors 
 $\vec \beta=(\beta_1,\beta_2,\beta_3) \in \R^3 $, perpendicular to $\vec 1=(1,1,1)$ and nondegenerate in the sense that
 \begin{equation}
\label{sepa}
\delta_0= \inf_{j\neq k} |\beta_j-\beta_k|>0.
\end{equation}
The adjoint bilinear  operators $T_{\vec \beta}$ to $\Lambda_{\vec \beta}$  are the  \emph{bilinear Hilbert transforms}, see for instance \cite{LT1,LT2} and the monograph \cite{ThWp}.  As noted in \cite{DoThiele15},    for a suitable choice of $\supp \widehat\phi$, the form  $ \Lambda_{\vec\beta}$ is   a nontrivial linear combination of the integral of the pointwise product   with a trilinear form involving the wave packet transforms:
\begin{equation}
\label{Vbeta}
\begin{split} &
\mathsf{V}_{\vec \beta} (f_1,f_2,f_3):= \int_{\mathcal Z} \left(\prod_{j=1}^3 G_j(u,t,\eta)\right) \, \d u \d t \d \eta, \\ & G_j(u,t,\eta):=F_{\phi}(f_j)(u,t,\alpha_j \eta + \beta_j t^{-1}) \end{split}
 \end{equation}
 where $\vec \alpha=(\alpha_1,\alpha_2,\alpha_3)$ is a unit vector orthogonal to $\vec \beta$ and $(1,1,1)$.
Therefore, H\"older-type estimates for $\Lambda_{\vec\beta}$ are an immediate consequence of the  corresponding  bounds   for $\mathsf{V}_{\vec \beta}$. Throughout the article, we refer to exponents triples $(p_1,p_2,p_3)$ with
\[
-\infty<p_1,p_2,p_3\leq \infty, \qquad \sum_{j=1}^3 {\textstyle \frac{1}{p_j}}=1
\]
 as H\"older triples of exponents.
The article \cite{DoThiele15} proves H\"older-type $L^{p_1}\times L^{p_2}\times L^{p_3}$-bounds for the form $\mathsf{V}_{\vec \beta}$
 in the \emph{local $L^2$} range $2<p_1,p_2,p_3<\infty$. These estimates recover    $L^p$ bounds for the bilinear Hilbert transform of  \cite{LT1} in the same restricted range, bypassing the discretization procedures which are ubiquitous in  the analysis of modulation invariant singular integrals; see \cite{DPTh2014,LT1,LT2,MTT}, the monographs \cite{MuscSchlII,ThWp} and references therein.
 
  The proof of \cite{DoThiele15} is articulated in two  distinct and complementary steps.
The first one is the application of an \emph{outer H\"older  inequality} (see Lemma \ref{lemmaredux})  to bound the trilinear integral in \eqref{Vbeta} by the product of certain \emph{outer $L^{p_j}$-norms} of each $G_j$, 
whose definition we postpone to Section \ref{sectents}. Here, we remark that this step does not rely in any way on  $G_j$ being  the wave packet transform of $f_j$. 
On the contrary, the nature of the $G_j$ is fundamental in the second and final step, which is the proof of an inequality of the type
\begin{equation}
\label{embintro}
\|G_j\|_{L^{p_j}(\mathcal Z, \sigma, \mathsf{s})} \leq C_{\vec\beta,p_j} \|f_j\|_{p_j},\qquad 2<p_j\leq\infty
\end{equation}
where the norm appearing on the left hand side is the same outer $L^{p_j}$-norm coming into play in the outer H\"older inequality. 
Estimates of this type are referred to as \emph{Carleson embedding theorems}. The motivation is that the $p_j=\infty$ case of this inequality is essentially the same as 
$$
|F_{\phi}(f_j)(u,t,\eta)^2| \frac{\d u\d t}{t} 
$$
being a Carleson measure on $Z=\R\times (0,\infty) $ when $f_j\in L^\infty$, uniformly in $\eta \in \R.$ 

  Carleson embeddings of the type \eqref{embintro} do not hold, at least in this form, for    $1<p_j<2$.  The main purpose of this article is to develop  a suitable extension of  \eqref{embintro}  outside of the local $L^2$ range treated in \cite{DoThiele15},  answering one of the questions posed in that article. A loose description of our main result, which is precisely  stated in Theorem \ref{bademb}, Section \ref{sectents}, is that for all $1<p<2$
 \[
\|\widetilde G_j\|_{L^{q}(\mathcal Z, \sigma, \mathsf{s})} \leq C_{\vec\beta,p,q}\|f_j\|_{p},\qquad p'<q\leq\infty
\]
where $\widetilde G_j$ is the restriction of  $ G_j$ outside a suitable exceptional subset of $\mathcal Z$ depending only on the maximal $p$-averages  of $f_j$.  The proof     involves a multi-frequency Cal\-de\-r\'on-Zygmund lemma which is an evolution of the approaches of \cite{DPTh2014,NOT}. 

Via a simple modification of the outer H\"older inequality argument, the Carleson embeddings of Theorem \ref{bademb}  yield     estimates for    $\mathsf{V}_{\vec \beta}(f_1,f_2,f_3)$, and thus for $\Lambda_{\vec \beta}(f_1,f_2,f_3)$   when $f_3$ is a suitably restricted subindicator, and $f_1,f_2$ are unrestricted functions, corresponding  to a \emph{direct proof} of the weak-type estimate
\begin{equation} \label{estimatell}
T_\beta: L^{p_1}(\R) \times L^{p_2}(\R) \to L^{\frac{p_1p_2}{p_1+p_2},\infty}(\R)
\end{equation}
for the adjoint(s) to $\Lambda_\beta$. This 
is an improvement from the generalized restricted weak-type paradigm of \cite{MTT}, where all input  functions in $\Lambda_{\vec \beta}$ are taken to be subindicator, resulting in the (weaker)   version of \eqref{estimatell} for subindicators. 
Our approach thus recovers, via bilinear weak-type interpolation rather than generalized restricted weak-type interpolation of e.g.\ \cite{MTT} or \cite[Chapter 3]{ThWp}, the strong-type bound for the bilinear Hilbert transforms $T_\beta$ in the full range  
\begin{equation}
\label{range} \textstyle  1<p_1,p_2\leq \infty, \qquad \frac23 <  \frac{p_1 p_2}{p_1+p_2}<\infty.
\end{equation}
of  expected exponents first established in \cite{LT2}.
 A precise statement and proof can be found in Theorem \ref{modelsumprop}, Section \ref{propsec}.

  We expect  Theorem \ref{bademb}, or suitable strengthenings thereof, to  impact on    other   related questions.  The first concerns uniform bounds for the family $\Lambda_{\vec\beta}$ as the degeneracy parameter $\delta_0$ from \eqref{sepa} approaches zero. In spite of the significant progress made  in the articles \cite{GrLi2004,LiX2006,ThUe}, it is still  open whether uniform bounds holds in the full expected range of exponents $p_1,p_2$ common  to both the generic and degenerate case.  The article \cite{OT} proves the full range of uniform bounds for a discrete model of $\Lambda_{\vec \beta}$, relying on a multi-frequency decomposition argument in the simple discrete setting, where the bad part has trivial contribution.
We do not explicitly track the  dependence on $\delta$ in our estimates. However, the argument for Theorem \ref{bademb} is currently ill-conditioned with respect to this parameter, and a more refined treatment is needed for uniform bounds.

The recent article \cite{DPOu} develops an outer $L^p$ theory for functions taking values in UMD Banach spaces, which is then used for the proof of  multilinear multiplier theorems in the  UMD-valued setting.
In view of possible applications to Banach-valued multilinear singular integrals with modulation invariance, it is reasonable to ask whether a form of Theorem \ref{bademb}  extends to    UMD Banach space valued functions (or interpolation UMD, as in \cite{HytLac2013}). In an appendix, we sketch how to obtain such an extension for Hilbert space valued functions.

 \subsection*{Notation} 
 We normalize the Fourier transform of   a Schwartz function $f\in \mathcal S(\R)$ as
$$
\widehat f(\xi) = \int_{\R} f(x) \e^{-i\xi x}\, \d x.
$$
Unless otherwise mentioned, the constants denoted by the letter  $C$, as well as those implied by the almost inequality signs, are meant to be absolute. Their   value may vary between occurrences without explicit mention. With $C_{a_1,\ldots, a_n}$, we denote a generic constant depending on the parameters $a_1,\ldots, a_n$ which may also differ at each occurrence.

We  write $c(I)$ for the center of the interval $I\subset \R$ and write $|I|$ for its length.    For $\kappa>0$, we denote by $\kappa I$ the interval centered at $c(I)$ with $|\kappa I|=\kappa I$.
 We say that $\cic{J}$ is  a finitely overlapping collection of intervals if  for all intervals $I \subset \R$
\begin{equation}
\label{finitover}
\sum_{J \subset I} |J| \leq C |I|.
\end{equation}
We will write
$$
\mathrm{M}_p  f(x) =\sup_{\substack{I \subset \R \textrm{ interval}\\ x \in I}}\left( \int_I |f(y)|^p \frac{\d y}{|I|}\right)^{\frac1p}
$$
for the $p$-Hardy-Littlewood maximal operator, for $1\leq p<\infty$ and omit the subscript when $p=1$.
If $\phi \in \mathcal S(\R)$, we will write $\mathcal{S}_{N}(\phi)$ for the Schwartz norms of $\phi$ of order $N$. 
\subsection*{Acknowledgments} 
We   are indebted to   Christoph Thiele for suggesting the problem of formulating an outer measure embedding theorem for the wave packet transform outside local $L^2$.  We are also  grateful to  the anonymous referee for his generous work which greatly helped with improving the clarity of the presentation.
\section{Main results}  \label{sectents}
 We begin by recalling from \cite{DoThiele15} the fundamental definitions  concerning the outer $L^p$ spaces appearing in the statement  of  our main Carleson embedding theorem.\subsection{Outer $L^p$ spaces from tents} 
Let $Z$ be a metric space and $\Tt$ be a subcollection of the   Borel subsets of $Z$. Following the terminology of \cite[Definition 2.1]{DoThiele15}, we refer to any function $\sigma:\Tt\to [0,\infty]$ as a \emph{premeasure}. The \emph{outer measure}    $\mu$ generated by $\sigma$ is
\[
Z\supset A \mapsto \mu(A):=\inf_{\Tt'}\sum_{\T\in \Tt'} \sigma(\T) 
\]
where the infimum is taken over all \emph{countable} subcollections $\Tt'$ of  $\Tt$ which cover the set $A$. 
Let $\mathcal{B}(Z)$ denote  complex, Borel-measurable functions on $Z$. A \emph{size} is  a map
\[
\mathsf{s}:\, \mathcal{B}(Z )\rightarrow [0,\infty]^{\Tt}
\]
that satisfies for any $F,G\in\mathcal{B}(Z )$ and $\T\in\Tt$  
\begin{align}
& \nonumber   |F|\leq |G| \implies  \mathsf{s}(F)(\T)\leq\mathsf{s}(G)(\T),\\
&\nonumber \mathsf{s}(\lambda F)(\T)=|\lambda|\mathsf{s}(F)(\T),\quad\forall \lambda\in\mathbb{C},\\
& \label{sizesubadd}  \mathsf{s}(F+G)(\T)\leq c_{\mathsf{s}}\mathsf{s}(F)(\T)+c_{\mathsf{s}}\mathsf{s}(F)(\T),\quad\text{for some fixed $c_{\mathsf{s}}\geq 1$}.
\end{align}  
We  refer to the triple $(Z, \sigma,\mathsf{s})$ as an \emph{outer measure space}. Notice that  the generating collection  $\Tt$ is implicitly encoded in the notation  as the domain of $\sigma$.   
We now introduce \emph{outer $L^p$ spaces}. Let us first define the \emph{outer (essential) supremum} of $F \in \mathcal{B}(Z)$ over the Borel-measurable subset $E\subset  Z$ as
\[
\underset{E}{\text{outsup}\,}\mathsf{s}(F):=\sup_{\T\in\Tt}\mathsf{s}(F\cic{1}_{ E})(\T).
\]
Then, for each $\lambda>0$, the \emph{super level measure} is defined by
$$
\mu(\mathsf{s}(F)>\lambda):= \inf \left\{ \mu(E): E\subset Z \textrm{ Borel}, \, \underset{Z\setminus E}{\text{outsup}\,}\mathsf{s}(F) \leq \lambda\right\}. 
$$  
Finally, for $0<p<\infty$ and   $f\in \mathcal{B}(Z )$  we set 
\begin{align*}
&\|F\|_{L^p(Z,\sigma, \mathsf{s})}:=\left(\int_0^\infty p\lambda^{p-1}\mu(\mathsf{s}(F)>\lambda)\,\d\lambda\right)^{\frac1p},\qquad 
\|F\|_{L^{p,\infty}(Z,\sigma,\mathsf{s})}:= \sup_{\lambda>0} \lambda\big(\mu(\mathsf{s}(F)>\lambda)\big)^{\frac1p},
\end{align*}
and for $p=\infty$
\[
\|F\|_{L^{\infty,\infty}(Z,\sigma, \mathsf{s})}=\|F\|_{L^\infty(Z,\sigma, \mathsf{s})}:=\underset{Z}{\text{outsup}\,}\mathsf{s}(F)=\sup_{
\T\in\Tt}\mathsf{s}(F)(\T).
\]
We are concerned with two interrelated concrete examples of outer $L^p$ spaces.

\subsubsection{Tents} Our base metric space will be the upper half-space  $Z=\R\times(0,\infty)$. Let $I\subset \R$ be an interval centered at $c(I)$.  The \emph{tent} over $I$ is defined as
\begin{equation}
\label{tentdef}
\T(I)= \{(u,t) \in Z: 0<t<|I|, |u-c(I)| <|I|-t\}.
\end{equation}
An outer measure   $\mu$ on $Z$, with $\{\T(I):I \subset \R\}$ as the  generating collection is then defined via the premeasure $\sigma(\T(I))=|I|$. 
For $F\in \mathcal{B}(Z)$, we define the \emph{sizes}
\begin{equation}
\label{size0} 
\begin{split} &
\mathsf{s}_q(F) (\T(I)):= \left(\frac{1}{|I|} \int_{T(I)} |F(u,t)|^q \, \frac{\d u \d t}{t} \right)^{\frac1q}, \qquad 0<q<\infty, \\ &\mathsf{s}_\infty(F) (\T(I)):=\sup_{(u,t) \in\T(I)} |F(u,t)|,
\end{split}
\end{equation}
and denote by
$
L^p(Z, \sigma, \mathsf{s}_q), \; L^{p,\infty}(Z, \sigma, \mathsf{s}_q)
$
the corresponding strong and weak outer $L^p$ space. Note that the case $p=\infty$ fits into the scale of \emph{tent spaces} originating in \cite{CMS}.

\subsubsection{Generalized tents} \label{ssgt} The outer measure space of generalized tents is  based on the upper $3$-half space $\mathcal Z= Z\times \R.$ Let $\cic{\alpha}$ be a set of parameters obeying the restrictions 
 \begin{equation}
 \cic{\alpha}=(\alpha,\beta,\delta), \qquad 
\label{paramt}
0<\alpha\leq 1, \quad 0\leq |\beta|< \textstyle \frac34, \quad 2^{-16} \delta_0\leq  \delta \leq 2^{-8} \delta_0,
\end{equation}
where $\delta_0$ is a small fixed constant. In fact, throughout the article we will work with $\delta_0$ being the same as in \eqref{sepa}.
Let $I\subset \R$ be an interval, $\T(I)$ be as in \eqref{tentdef}  and $\xi \in \R$.  The corresponding \emph{generalized  tent}  with parameters $\cic{\alpha}$, and its \emph{lacunary part},  are defined by 
\[
\begin{split}&
\T_{\cic \alpha}(I,\xi):= \{(u,t,\eta) \in \mathcal Z: (u,t) \in \T_{\cic \alpha}(I), |\alpha(\eta-\xi)+\beta t^{-1}| \leq t^{-1}\}, \\ & \T^\ell_{\cic \alpha}(I,\xi):= \{(u,t,\eta) \in \T_{\cic \alpha}(I,\xi): t|\xi-\eta|>\delta\}.
\end{split}
\]
Outer measures $\nu_{\cic{\alpha}}$ on    $ \mathcal Z$, with $\{\T_{\cic \alpha}(I,\xi):I \subset \R, \xi \in \R\}$ as the generating collection are then defined via the premeasure $\sigma_{\cic{\alpha}}(\T_{\cic \alpha}(I,\xi))=|I|$. 
For $F\in\mathcal{B}(\mathcal Z)$,  we define the   sizes
\begin{equation}
\label{size} 
\begin{split}&
\mathsf{s}_2(F) (\T(I,\xi)):= \left(\frac{1}{|I|} \int_{\T^\ell_{\cic \alpha}(I,\xi)} |F(u,t,\eta)|^2 \,  \d u \d t\d \eta  \right)^{\frac12},  \\ &\mathsf{s}_\infty(F) (\T(I,\xi)):=\sup_{(u,t,\eta) \in\T_{\cic \alpha}(I,\xi)} |F(u,t,\eta)|, \\ &  \mathsf{s}:=\mathsf{s}_2+ \mathsf{s}_\infty.\end{split}
\end{equation} 
We denote by
$
L^p(\mathcal Z, \sigma_{\cic \alpha}, \mathsf{s}), \; L^{p,\infty}(\mathcal Z, \sigma_{\cic \alpha}, \mathsf{s})
$
and similarly for the other sizes, the corresponding strong and weak outer $L^p$ spaces.

\subsection{A Carleson embedding below local $L^2$}

Let   $\phi\in \mathcal S(\R)$ with $\supp \widehat \phi \subset [-2^{-8}\delta , 2^8\delta]$, where $\delta$ refers to \eqref{paramt}. In \cite{DoThiele15}, it was proved that  the wave packet transform $F_\phi(f)$  defined in \eqref{defFphi} enjoys the local $L^2$ Carleson embeddings 
\begin{align}&\label{goodemb1} \|F_\phi (f) \|_{L^q(\mathcal Z, \sigma_{\cic{\alpha}},\mathsf{s})} \leq C_{\alpha,\phi,q}
\|f\|_q, \qquad 2<q\leq \infty, \\ & \label{goodemb2}
\|F_{\phi } (f)\|_{L^{2,\infty}( \mathcal Z, \sigma_{\cic{\alpha}},\mathsf{s} )} \leq C_{\alpha,\phi,2} \|f\|_2,  
\end{align}  with $C_{\alpha,\phi,q}=C_q\alpha^{-C} \mathcal{S}_C(\phi).$
Our main result is  a suitable extension of \eqref{goodemb1}    below local $L^2$.  
\begin{theorem} \label{bademb} Let $f\in \mathcal S( \R)$,  $\lambda>0$, $1<p<2$. Define 
 \begin{equation}  \label{maxfctbd} \begin{split}  &
\cic{I}_{f,\lambda,p}:= \textrm{\emph{maximal\ dyadic\ intervals }} I\textrm{\emph{ s.t. }}
 I\subset \left\{x\in\R: \mathrm{M}_p(\mathrm M f)(x)>{\lambda\|f\|_p}\right\},\\  & E_{f,\lambda,p}:=\bigcup_{I\in \cic{I}_{f,\lambda,p}}  { \T}(3I) \times \R  \subset \mathcal Z. \end{split}
 \end{equation}
 Then
 \begin{equation} \label{cet} \left\|F_{\phi } (f) \cic{1}_{\mathcal Z\setminus  E_{f,\lambda,p}}\right\|_{L^{q} (\mathcal Z, \sigma_{\cic{\alpha}},\mathsf{s})} \leq  C_{p,q,\alpha,\phi} \lambda^{1-\frac{p}{q}}{\|f\|_p}\qquad \forall p'<q\leq\infty,
\end{equation}
The constant $C_{p,q,\alpha,\phi}$ depends on $   p$, $q$, $\alpha$ (polynomially) and $\phi$ only.
\end{theorem}
We use  Theorem \ref{bademb}, coupled with the outer H\"older inequality,   to extend the $L^p$ estimates for $\mathsf{V}_{\vec \beta}$ to exponents outside local $L^2$.
\begin{theorem} \label{modelsumprop}Let $(p_1,p_2,p_3)$ be a H\"older tuple with $p_1,p_2$ within the range \eqref{range}.
 For all  $A\subset \R$ of finite measure, and  all   $f_1,f_2\in \mathcal S(\R)$, there exists   $\widetilde A\subset A$ with $|A| \leq 2 |\widetilde A|$ and 
\[
 \big|\mathsf{V}_{\vec \beta} (f_1,f_2,f_3)\big| \leq C_{\vec\beta, p_1,p_2, } \|f_1\|_{p_1} \|f_2\|_{p_2} |A|^{ \frac{1}{p_3} } \qquad \forall f_3 \in \mathcal S(\R),\, |f_3| \leq \cic{1}_{\widetilde A}.
\]
\end{theorem}
Relying on the observation above \eqref{Vbeta}, the same estimate of Theorem \ref{modelsumprop} holds for the forms $\Lambda_{\vec \beta}$. Equivalently, the adjoints $T_{\vec \beta}$ to $\Lambda_{\vec \beta}$ satisfy \eqref{estimatell} for all $p_1,p_2$ within the range \eqref{range}.
Finally, since the range \eqref{range} is open, multilinear interpolation upgrades \eqref{estimatell} to the strong-type estimate.
\begin{remark}
\label{sketchpf} The proof of Theorem \ref{bademb} contains a modulation invariant version of the Cal\-der\'on-Zygmund decomposition. Loosely speaking, the argument proceeds by selecting the generalized tents on which $F_\phi(f)$ has large size, and decomposing $f$ into the sum of a \emph{good part} $g\in L^2$ and a \emph{bad part} $b$. The bad part is devised to have moments of rather high order vanishing when integrated against the selected frequencies. This extra cancellation renders the  size of  $F_\phi(b)$     very small, forcing $F_\phi(g)$ to have  large size. At this point, since $g\in L^2$, we get to use the weak-$L^2$ version of the Carleson embedding \eqref{goodemb2}, or rather its proof. Unlike \cite{DPTh2014,NOT} and the similar projection argument used in the Walsh context (see \cite{DP2} and references therein), the frequencies on which $f$ is projected are intrinsic to $f$ itself. 
\end{remark}

 \subsection{Proofs and structure of the article} In the upcoming Section \ref{propsec}, we prove Theorem \ref{modelsumprop}. The main steps of the proof of Theorem \ref{bademb}   are postponed to Section \ref{secgentents}. The proof relies on the material of Section \ref{secembzero}, which contains somewhat refined Carleson embedding theorems for the wave packet transform of $f$ at a fixed frequency $\xi$.
 The preliminary Sections \ref{secoutlp} and \ref{adapsec} are respectively dedicated to some properties of outer $L^p$ spaces  and to   some results on ad\-apted systems of wave packets, which are   used in  Section \ref{secembzero}.

\section{Estimating the trilinear form $\mathsf{V}_{\vec \beta}$ via Carleson embeddings} \label{propsec} 
This section is dedicated to the proof of Theorem \ref{modelsumprop}. Before the main argument, we state two preliminary results.
The following lemma, which  summarizes the discussion in \cite[p.\ 46]{DoThiele15}, is the above mentioned H\"older inequality involving outer $L^p$ norms on the upper 3-half space.  \begin{lemma} \label{lemmaredux} Let   $\{\vec \alpha,  \vec{\beta}\}$  be an orthonormal basis of $(1,1,1)^\perp$, $\delta_0$ be as in \eqref{sepa}, 
and $\cic{\alpha}_j:=(\alpha_j,\beta_j,2^{-9}\delta_0),$  $j=1,2,3.$ 
 Then $$
\int_{\mathcal Z}\left( \prod_{j=1}^3  |F_j(u,t,\alpha_j \eta+ \beta_j t^{-1})| \right)\, \d u\d t \d \eta \leq C_{\delta_0} \prod_{j=1}^3  \|F_j\|_{L^{p_j} (\mathcal {Z} , \sigma_{\cic{\alpha}_j}, \mathsf s)}
$$
for all   $F_j\in \mathcal B(\mathcal Z)$, $j=1,2,3$,  and  H\"older tuples $(p_1,p_2,p_3)$  with $1\leq p_1,p_2,p_3\leq \infty$. 
\end{lemma}
We will also need a localized version of \eqref{goodemb1} to be applied in conjuction with Theorem \ref{bademb}. The proof   is given in Section \ref{secgentents}. \begin{proposition} \label{goodemb} Let $f\in \mathcal S( \R)$ and $K>1$ be fixed. Assume $\cic{J}$ is  a finitely overlapping collection of intervals, see \eqref{finitover},  
  such that 
\begin{equation}
\label{goodemb3a}
\supp f \cap KJ = \emptyset \qquad \forall J \in \cic{J}.
\end{equation}
Then, for all $R>1$ and $2<q\leq \infty$
\begin{equation}
\label{goodemb3}
\|F_{\phi } (f)\cic{1}_{E_{\cic{J}}}\|_{L^{q}( \mathcal Z, \sigma_{\cic{\alpha}},\mathsf{s} )} \leq C_{\cic{\alpha},\phi,R,q} K^{-R} \|f\|_q,  \qquad E_{\cic{J}}:= \bigcup_{J \in \cic{J}} { \T}(3J) \times \R  \subset \mathcal Z.
\end{equation}
\end{proposition}

\begin{proof}[Proof of Theorem \ref{modelsumprop}]  
The implicit constants in this proof are allowed to depend on $p_1,p_2$, $\phi$, $\vec \beta$ without explicit mention. We deal in detail with the harder case $1<p_1,p_2<2$. When either or both $p_1>2,p_2>2$ , the estimate of the theorem is obtained by a similar (in fact, simpler) argument where \eqref{goodemb1} is used in place of the main estimate \eqref{cet} for the corresponding  function $f_j$.    The case where either or both $p_1=2,p_2=2$ then follows by interpolation, see for instance \cite{MuscSchlII}.
We will use later that, since
$1/p_1' + 1/{p_2'} > 1/2
$
we can find a tuple $q_1,q_2,q_3$ satisfying
\[\textstyle \frac{1}{q_1}+\frac{1}{q_2}+ \frac{1}{q_3}=1, \qquad 
q_1>p_1',\; q_2 >p_2',\; q_3>2.
\] which is fixed from now on.
By linearity in $f_1,f_2$ and horizontal scaling invariance, that is by possibly replacing $f_j$ with
$
|A|^{-\frac{1}{p_j}}f_j(|A|^{-1} \cdot)
$   we may work with
$$
\|f_1\|_{p_1}= \|f_2\|_{p_2}= |A|=1.
$$

We begin the actual proof. To construct $\widetilde A$, we find $C_0$ large enough such that each
$$
H_j:=\{ x \in\R: \mathrm{M}_{p_j} (\mathrm{M}f_j)(x) > C_0\}, \qquad j=1,2
$$
has measure less than $1/4$ and set $\widetilde A= A \cap (H_1\cup H_2)^c$. 
Let us define, with reference to the notation of \eqref{maxfctbd}, and for $k=0,1,\ldots,$
$$
\cic{I}_k:= \bigcup_{j=1,2} \cic{I}_{f_j, 2^{10+k} C_0, p_j}, \qquad  E_k:=\bigcup_{j=1,2} E_{f_j, 2^{10+k} C_0, p_j}, \qquad \widetilde{E_k}= E_{k} \setminus E_{k+1}.
$$
Let $f_3$ be supported on $\widetilde A$ and bounded by 1 be fixed from now on. For reasons of space, let us write  
$
G_j $, $j=1,2,3$  as in the second line of \eqref{Vbeta}. We have, with $ z=( u,t,\eta)$
\begin{equation}\label{split0}
\big|\mathsf{V}_{\vec \beta} (f_1,f_2,f_3)\big| \leq \int_{\mathcal Z}  \left(\prod_{j=1}^3 |G_j \cic{1}_{\mathcal Z\setminus E_0} (z)|\right) \, \d z+ \sum_{k=0}^\infty \int_{\mathcal Z}  \left(\prod_{j=1}^3 |G_j\cic{1}_{\widetilde{E_k}}(z)|\right) \, \d z.
\end{equation} We will now prove that \eqref{split0} $\lesssim 1$.
  Let $\cic{\alpha}_j=(\alpha_j,\beta_j, \delta)$,  $j=1,2,3$, be our parameter vectors with $\vec \alpha=(\alpha_1,\alpha_2,\alpha_3)$ chosen as in Lemma \ref{lemmaredux}.
By  Theorem \ref{bademb}, estimate \eqref{cet} for  $q_j$, $j=1,2$ we have the embeddings
\begin{equation}
\label{strtype1}
\|F_\phi (f_j)\cic{1}_{\mathcal Z\setminus E_k}\|_{L^{q_j}(\mathcal Z,\sigma_{\cic{\alpha}_j},\mathsf{s})} \leq \left\|F_\phi (f_j)\cic{1}_{\mathcal Z\setminus E_{f_j, 2^{10+k} C_0, p_j}}\right\|_{L^{q_j}(\mathcal Z,\sigma_{\cic{\alpha}_j},\mathsf{s})} \lesssim 2^k
\end{equation} 
To control $F_\phi (f_3)$ we use  \eqref{goodemb1} and \eqref{goodemb3}, with $R=10$ say, and obtain the estimates
 \begin{equation}\label{strtype4}
 \begin{split} &
 \|F_\phi (f_3)\|_{L^{q_3}(\mathcal Z,\sigma_{\cic{\alpha}_3},\mathsf{s})} \lesssim \|f_3\|_{q_3}\leq |A|^{\frac{1}{q_3}}= 1,\\
&\|F_\phi (f_3)\cic{1}_{E_k}\|_{L^{q_3}(\mathcal Z,\sigma_{\cic{\alpha}_3},\mathsf{s})} \lesssim 2^{-10k} \|f_3\|_{q_3}\leq 2^{-10k}  \end{split} \end{equation}
We remark that the assumptions of \eqref{goodemb3} are met since    the support of $f_3$ does not intersect $3\cdot 2^k I$ for all $I \in \cic{I}_k$.  
Using Lemma \ref{lemmaredux}, and later \eqref{strtype1}  and the first estimate of \eqref{strtype4}    the first term in \eqref{split0} is  bounded  by 
\[
\left(\prod_{j=1}^2  \|F_\phi (f_j)\cic{1}_{\mathcal Z\setminus E_0}\|_{L^{q_j}(\mathcal Z,\sigma_{\cic{\alpha}_j},\mathsf{s})} \right) \|F_\phi (f_3)\|_{L^{q_3}(\mathcal Z,\sigma_{\cic{\alpha}_3},\mathsf{s})} \lesssim 1.
\]
  Again using Lemma \ref{lemmaredux}, and later \eqref{strtype1}, trivial estimates and the second bound of \eqref{strtype4}    the second term in \eqref{split0} is controlled by 
\[\begin{split} & \quad 
 \prod_{j=1}^3  \|F_\phi (f_j)\cic{1}_{\widetilde{E}_k}\|_{L^{q_j}(\mathcal Z,\sigma_{\cic{\alpha}_j},\mathsf{s})} \\ & \leq \left(  \prod_{j=1}^2  \|F_\phi (f_j)\cic{1}_{\mathcal Z \setminus  E_{k+1}}\|_{L^{q_j}(\mathcal Z,\sigma_{\cic{\alpha}_j},\mathsf{s})}\right) \|F_\phi (f_3)\cic{1}_{E_k}\|_{L^{q_3}(\mathcal Z,\sigma_{\cic{\alpha}_3},\mathsf{s})}  \lesssim 2^{-8k}
\end{split}
\]
This is summable in $k$, completing the estimation of \eqref{split0} and the proof of the Theorem.
\end{proof}


\section{Elementary properties of outer $L^p$} \label{secoutlp}  In this short section, we collect, without proofs, some  basic properties of outer $L^p$ spaces which we   use in our analysis. The following proposition summarizes \cite[Proposition 3.1, 3.3]{DoThiele15}.
\begin{proposition}\label{propoutermeassumm} Let $(Z,{\sigma}, \mathsf{s})$ be an outer measure space, $F,G \in \mathcal B(Z)$, and     $0<p\leq \infty$. Then
\begin{align}
& \nonumber  |F|\leq |G| \implies \|F\|_{L^p(Z,\sigma,\mathsf{s})}\leq \|G\|_{L^p(Z,\sigma,\mathsf{s})},\\
&\nonumber \|\lambda F\|_{L^{p}(Z,{\sigma}, \mathsf{s})}=|\lambda| \|  F\|_{L^{p}(Z,{\sigma}, \mathsf{s})},  \qquad\forall \lambda\in\mathbb{C},
\\
&\nonumber\|  F\|_{L^{p}(Z,\lambda{\sigma}, \mathsf{s})}=\lambda^{1/p} \|  F\|_{L^{p}(Z,{\sigma}, \mathsf{s})},  \qquad\forall \lambda>0,\\
& \label{Lpsubadd} \|  F+G\|_{L^{p}(Z, {\sigma}, \mathsf{s})}\leq 2c_{\mathsf{s}}\left( \|  F\|_{L^{p}(Z, {\sigma}, \mathsf{s})}+ \|  G\|_{L^{p}(Z, {\sigma}, \mathsf{s})}\right),\\
\label{Lplog}\begin{split} &  \|F\|_{L^p(Z,\sigma,\mathsf{s})}\leq c_{p,p_1,p_2}\left(\|F\|_{L^{p_1,\infty}(Z,\sigma,\mathsf{s})}\right)^{\theta}\left(\|F\|_{L^{p_2,\infty}(Z,\sigma,\mathsf{s})}\right)^{1-\theta}, \\   &
0<p_1<p<p_2\leq\infty, \quad \textstyle \frac{1}{p}=\frac{\theta}{p_1}+ \frac{1-\theta}{p_2}.
\end{split}
\end{align}
 And identical statements hold for the spaces $L^{p,\infty}$.
\end{proposition}
\begin{remark}
 Iterating \eqref{Lpsubadd}, one obtains the quasi-triangle inequality
\begin{equation}
\label{trianin}
\left\|\sum_{k=0}^K F_k\right\|_{L^p(Z,\sigma,\mathsf{s})} \leq\sum_{k=0}^K  (2 c_{\mathsf{s}})^{(k+1)} \left\| F_k\right\|_{L^p(Z,\sigma,\mathsf{s})}.
\end{equation}
Note that the constant $c_{\mathsf s}$ in \eqref{Lpsubadd} is the same appearing in \eqref{sizesubadd} and can be taken equal to 1 in our concrete cases. 
\end{remark}
We also record the Marcinkiewicz interpolation theorem in outer $L^p$, \cite[Proposition 3.5]{DoThiele15}
\begin{proposition}  \label{propint} Let $(X,\nu)$ be a measure space,  $(Z,\sigma, \mathsf{s})$ be an outer measure space and $T$ be a quasi-sublinear operator mapping functions in $L^{p_1}(X,\nu)$ and $L^{p_2}(X, \nu)$ into $\mathcal{B}(Z)$, for some $0< p_1<p_2\leq\infty$. Assume that for all $f\in L^{p_1}(X, \nu)+L^{p_2}(X, \nu)$,
\[
\|T(f)\|_{L^{p_j,\infty}(Z,\sigma, \mathsf{s})}\leq B_j\|f\|_{L^{p_j}(X, \nu)}, \qquad j=1,2.
\]
Then, for $0<\theta<1$,
\[
\|T(f)\|_{L^p(Z,\sigma, \mathsf{s})}\leq C_{\theta,p_1,p_2} B_1^{\theta}B_2^{1-\theta}\|f\|_{L^p(X, \nu)}, \qquad   \textstyle \frac{1}{p}=\frac{\theta}{p_1}+ \frac{1-\theta}{p_2}.
\]

\end{proposition}
\begin{remark}[Scaling properties of estimate \eqref{cet}] \label{scalingrem} The structure of the wave packet transform \eqref{defFphi} and of the outer $L^q(\mathcal Z, \sigma_{\cic{\alpha}},\mathsf{s})$ spaces yield a useful horizontal  scaling property for estimate \eqref{cet}.
Fix $1\leq p<\infty$,   $f\in L^p(\R)$ and $\kappa>0$ and write, referring to \eqref{maxfctbd},
\[
f_{\kappa}:= \mathrm{Dil}^{p}_{\kappa^{-p}} f= \kappa  f(\kappa^{p}\cdot), \qquad F:= F_{\phi } (f) \cic{1}_{\mathcal Z\setminus  E_{f,\lambda,p}}, \qquad F_\kappa:= F_{\phi } (f_\kappa) \cic{1}_{\mathcal Z\setminus  E_{f_\kappa,\kappa\lambda,p}}
\]
Then $\|f_\kappa\|_p=\|f\|_p$ and 
\begin{equation}
\label{scalingest}
\|F\|_{L^q(\mathcal Z, \sigma_{\cic{\alpha}},\mathsf{s})}= \kappa^{\frac pq-1}  \|F_\kappa\|_{L^q(\mathcal Z, \sigma_{\cic{\alpha}},\mathsf{s})}
\end{equation}
and similarly for outer weak $L^q$ norms.
To see this, introduce the bijection \[
T_\kappa:\mathcal Z\to \mathcal Z,\qquad  T_\kappa(u,t,\eta)=(\kappa^{-p}u,\kappa^{-p}t,\kappa^p\eta)\]
Then \cite[Proposition 3.2]{DoThiele15} applied to both $T_\kappa$ and its inverse yields the equality of norms \[
  \|G\circ T_\kappa\|_{L^q(\mathcal Z, \sigma_{\cic{\alpha}},\mathsf{s})}=\kappa^{\frac pq}\|G\|_{L^q(\mathcal Z, \sigma_{\cic{\alpha}},\mathsf{s})}\]
and similarly for weak norms.
We omit the straightforward verification of the assumptions of the proposition for $T_\kappa$.
Horizontal scaling  of the definition of \eqref{maxfctbd} yields  
\[
 (u,t,\eta) \in E_{f , \lambda,p}   \iff T_\kappa(u,t,\eta)\in   E_{f_\kappa,\kappa\lambda,p}
\]
while the structure of \eqref{defFphi} entails 
\[
F_{\phi } (f)(u,t,\eta)=\kappa^{-1}F_{\phi } (f_\kappa)\left(T_\kappa(u,t,\eta)\right),
\]
and we conclude that $F=\kappa^{-1} F_\kappa \circ T_\kappa $, whence \eqref{scalingest}.
\end{remark}
\section{Adapted systems}
\label{adapsec}

 Let $\xi \in \R$ and $0<T\leq \infty$. A family  
$$
{\Phi}^{\xi}:=\{\phi_t: t \in (0,T) \} \subset \mathcal S(\R)
$$
is said to be a $\xi$-\emph{adapted system}  with \emph{adaptation constants} $A_{N}= A_{N}( {\Phi}^{\xi})$ if
\begin{equation}
\label{adaptxi}
\sup_{\substack{n,m\leq N }}\sup_{t \in (0,T)} \sup_{x\in \R} t^{m+1} \left(   1+\left| \frac{x}{t} \right|\right)^n \left|   (\e^{-i\xi \cdot} \phi_t(\cdot))^{(m)}(x) \right| \leq A_{N}
\end{equation}
for all nonnegative integers $N$. The $\xi$-adapted system ${\Phi}^{\xi}$ is said to \emph{have mean zero}  if
\begin{equation}
\label{adaptxizero}
 \widehat{\phi_t} (\xi) =0 \qquad \forall t \in (0,T).
\end{equation}
  A prime example of $\xi$-adapted system [resp.\ with mean zero] is obtained by  dilation and modulation of a single   $\phi \in \mathcal S(\R)$ [resp.\ with mean zero]
\begin{equation}
\label{dilmod}
\phi_t:= \mathrm{Mod}_\xi \mathrm{Dil}_t^1 \phi,\qquad
\phi_t(x):= \e^{i\xi x} \frac{1}{t} \phi\left( \frac{x}{t}\right), \qquad t \in (0,\infty).
\end{equation}

The remainder of the section is occupied by three lemmata, each of which plays a significant role in our investigation. First, we record the following well-known principle of the Littlewood-Paley theory.
\begin{lemma}
\label{areaintegral}
Let ${\Phi}^{\xi}$ be a $\xi$-adapted system with mean zero. Then 
$$
\left(  \int _{  \R\times (0,\infty)} |f\ast \phi_t (u)|^2 \, \frac{\d u \d t}{t} \right) \lesssim   A_{3}({\Phi}^{\xi}) \|f\|_2.
$$
\end{lemma}
\begin{proof}By replacing $f$ with $f(\cdot) \e^{-ix \xi}$, we can reduce to the case $\xi=0$. A consequence of \eqref{adaptxi}-\eqref{adaptxizero} is that
$$
|\widehat{\phi_t} (\xi)| \lesssim  A_{3}\min\{t|\xi|,(1+t|\xi|)^{-1}\}.
$$ 
This entails  
$$
\sup_{\xi} \int_{0}^{\infty} |\widehat{\phi_t} (\xi)|^2 \frac{\d t}{ t }   \lesssim   A_{3} ,
$$
from which the lemma follows by two   applications of Plancherel. 
\end{proof}
The next two lemmata deal with the decomposition  of the functions belonging to a   $\xi$-adapted system  into a compactly supported part and an exponentially small remainder. 
   \begin{lemma}
\label{lemmasplit}
Let ${\Phi}^{\xi}$ be a $\xi$-{adapted} system [resp.\ with mean zero]. Let $K\geq 1$ and $Q,N\in \mathbb N$ be given. There exist   two   $\xi$-adapted systems [resp.\ with mean zero] ${\Psi}^{\xi}=\{\psi_t\}, {\Upsilon}^{\xi}=\{\upsilon_t\}$,   such that
\begin{align} & \label{lemmasplit0}
\phi_t = \psi_t+ K^{-Q} \upsilon_t \qquad \forall t \in (0,T),\\ \nonumber  &A_{N}({\Psi}^{\xi}),A_{N}({\Upsilon}^{\xi}) \leq C_{Q,N}  A_{N+Q+1}( {\Phi}^{\xi})\qquad  N \in \mathbb N,  
\\ &  \supp \psi_t \subset [-Kt,Kt] \qquad \forall t \in (0,T).
\nonumber
\end{align}
\end{lemma}
\begin{proof}  We only prove the mean zero case, which is more involved. By replacing $\phi_t$ with  $\e^{-i\xi \cdot} \phi_t(\cdot)$, we   can reduce to the case $\xi=0$ and write ${{\Phi}} $ instead of ${{\Phi}^{0}}$. Note that, for any fixed value $t_0>0$, the system $$\left\{\widetilde{ \phi_t}:= \textstyle \frac{1}{t_0}  \phi_{\frac{t}{t_0}}\left( \textstyle  \frac{\cdot}{t_0}\right): t \in (0,t_0 T)\right\}$$ 
is $0$-adapted with the same constants as ${{\Phi}}$, 
 it suffices to produce the decomposition \eqref{lemmasplit0} for $t=1$. We write $\phi$ in place of $\phi_1$, and similarly for $\psi_1, \upsilon_1$.  
 
 Let $\beta $ be a smooth function satisfying 
\[
|\beta| \leq 1, \qquad \beta\equiv 1 \, \textrm{ on } \left[\textstyle -\frac12,\frac12\right],\qquad  \supp \beta \subset [-1,1] \qquad \int \beta=1.
\]
Let $K>1$. Define $\beta_{K}= \beta(\cdot/K)$. We record the obvious facts 
\begin{equation}
\label{lemmasplit1}
\begin{split}
&\supp \beta_{K} \subset [-K,K], \qquad  \supp (\cic{1}_\R- \beta_{K}) \subset \left[\textstyle -\frac{K}2,\frac{K}2\right]^c,\\
& \sup_{x \in \R} \left( \textstyle 1+\left| x \right|\right)^n \left|     \beta_K^{(m)}(x) \right| \leq C_{n,m} K^{n-m}
\end{split}
\end{equation}with constant $C_{n,m}$ depending only on $m,n \in \mathbb N$ through $\beta$'s derivatives. Define
$
I_{K}:= \int_\R \beta_{K} \phi\, \d x.
$
In view
 of the first line of \eqref{lemmasplit1}, since $\phi$ has mean zero, we have for all nonnegative integers $Q$
\begin{equation}
\label{lemmasplit2}
\begin{split}&\quad 
\left|I_{K} \right| = \left|  \int_\R( \cic{1}_\R - \beta_{K}) \phi\, \d x\right| \leq   \int_{|x|>\frac{K}{2}}| \phi|\, \d x \\ &  \leq   A_{Q+1}( {\Phi} )   \int_{|x|>\frac{K}{2}} (1+|x|)^{-(Q+1)} \, \d x \leq C_Q A_{Q+1}( {\Phi}^{\xi})  K^{-Q}.
\end{split}
\end{equation}
Fix now $K$ and $Q$. The decomposition \eqref{lemmasplit0} is achieved by setting 
\[
\begin{split}& K^{-Q}\upsilon:=\frac{I_{K}}{K} \beta_K + \eta_K, \qquad \eta_K:=(\cic{1}_\R-\beta_K) \phi, \\
& \psi:= \beta_K \phi - \frac{I_{K}}{K} \beta_K
\end{split}
\]
Indeed, $\psi $ is obviously supported in $[-K,K].$ Exploiting \eqref{lemmasplit1}, and \eqref{lemmasplit2} with $Q+ N $ in place of $Q$, we have for $ m,n\leq N$
\begin{equation}
\label{lemmasplit4}
\sup_{x \in \R} \left( \textstyle 1+\left| x \right|\right)^n \left|     \frac{I_{K}}{K} \beta_K^{(m)}(x) \right| \leq C_{Q,N} A_{Q+N+1}( {\Phi} )K^{-Q}.
\end{equation}
Furthermore, for  $m,n\leq N$, since $\beta_K^{(j)}$ is supported in $|x|>K/2$ for $j\geq 1$,
\begin{equation}
\label{lemmasplit5} \begin{split}&
  \left|       \eta_K^{(m)} (x) \right| \leq |(\cic{1}_\R-\beta_K)(x)| |\phi^{(m)}(x)|+\sum_{j=1}^{m}   {m \choose j} |\beta_K^{(j)}(x)||\phi^{(m-j)}(x)|\\ & \leq C_{Q,N}  A_{N+Q}( {\Phi} ) K^{-Q} \left( \textstyle 1+\left| x \right|\right)^{-n}.  \end{split}   \end{equation}
Combining \eqref{lemmasplit4}-\eqref{lemmasplit5}, we obtain the required adaptation bound   for $\upsilon$, and the one for $\psi$ follows by comparison. 
\end{proof}
Iterating the proof of the previous lemma we obtain the following decomposition of a $\xi$-adapted system ${\Phi}^{\xi}$ into $\xi$-adapted systems with compact support. The statement is a slightly more precise version of \cite[Lemma 3.1]{MPTT06}, to which we send for the proof details.
\begin{lemma}
\label{lemmasplititer}
Let ${\Phi}^{\xi}$ be a $\xi$-{adapted} system [resp.\ with mean zero] and $Q\in \mathbb N$.  There exists     $\xi$-adapted systems [resp.\ with mean zero] ${\Psi}^{\xi;k}=\{\psi_{t;k}\}$,  $k=0,1,2\ldots$   such that
\begin{align}\nonumber &
\phi_t = \sum_{k\geq 0} 2^{-Qk}\psi_{t;k} \qquad   \forall t \in (0,T),\\  &A_{N}( {\Psi}^{\xi;k}) \leq C_{Q,N} A_{Q+N+1}( {\Phi}^{\xi})\qquad N\in \mathbb N,  \label{lemmasplititer2}
\\ & \supp \psi_{t;k} \subset [-2^{k}t,2^{k}t] \qquad \forall t \in (0,T).\nonumber
\end{align}
\end{lemma}

\section{Carleson embeddings for tents revisited} \label{secembzero}
This section is dedicated to   Carleson embedding theorems in the outer $L^p(Z,\sigma, \mathsf{s}_q)$ spaces previously defined,  for the function
\begin{equation}
\label{Phixi}
{\Phi}^\xi (f) (u,t)= f* \phi_t(u), \qquad (u,t) \in Z
\end{equation}
where ${\Phi}^\xi=\{\phi_t: t \in (0,\infty) \}$ is a $\xi$-adapted system.   Throughout this section,   the absolute constant  $C$ appearing in  $A_{C}({\Phi}^\xi )$ is moderate; $C=20$ would suffice.

\subsection{Global-type Carleson embeddings} \label{ssglob}
The upcoming Carleson embedding theorems are very close in spirit to those given in \cite[Section 4]{DoThiele15}. The slight differences with \cite{DoThiele15} is that we do not work with dilates of a single function $\phi$ as in \eqref{dilmod}, and that our $\phi_t$ need not be compactly supported. Albeit these additions are minor, they will allow us to obtain generalized $L^\infty$ Carleson embeddings in Section \ref{secgentents} by simply averaging the results of this section.  
 \begin{proposition} \label{cemb} Let ${\Phi}^\xi$ be a $\xi$-adapted system. Then
\begin{align} &
\label{LpSinfty} \|{\Phi}^\xi (f) \|_{L^p(Z,\sigma,\mathsf{s}_\infty)} \lesssim A_{C}({\Phi}^\xi ) \|f\|_p, \qquad 1<p\leq \infty, 
\\ & \|{\Phi}^\xi (f) \|_{L^{1,\infty}(Z,\sigma,\mathsf{s}_\infty)} \lesssim A_{C}({\Phi}^\xi )  \|f\|_1.
\label{L1Sinfty}
\end{align}
If, in addition, ${\Phi}^\xi$  has mean zero,
\begin{align} &
\label{LpS2} \|{\Phi}^\xi (f) \|_{L^p(Z,\sigma,\mathsf{s}_2)} \lesssim A_{C}({\Phi}^\xi ) \textstyle  \|f\|_p, \qquad 1<p\leq \infty,
\\ & \|{\Phi}^\xi (f) \|_{L^{1,\infty}(Z,\sigma,\mathsf{s}_2)} \lesssim A_{C}({\Phi}^\xi )  \|f\|_1.
\label{L1S2}
\end{align}
\end{proposition}
After a couple of remarks,  we provide the details of proof for \eqref{LpS2} and \eqref{L1S2}. The proofs of \eqref{LpSinfty} and \eqref{L1Sinfty} can be then easily readapted (and are almost immediate to begin with). In the upcoming remark, we reduce to compact support of $\phi_t\in \Phi^\xi$ by means of Lemma \ref{lemmasplititer}.
\begin{remark}
\label{remcptsupp} 
We claim that  the estimates of Proposition \ref{cemb} can be obtained by proving the corresponding version under the assumption that $\phi_t \in {\Phi}^\xi$ is supported in $[-Kt,Kt]$, with a bound depending  polynomially on $K$ (and on the adaptation constants of ${\Phi}^\xi$). For instance, \eqref{LpS2}-\eqref{L1S2} will follow from 
\begin{align} \label{LpS2cpt}&
\|{\Phi}^\xi (f) \|_{L^p(Z,\sigma,\mathsf{s}_2)} \lesssim \textstyle   K^2 A_{3}({\Phi}^\xi )\|f\|_p, \qquad 1<p\leq \infty, \\  \label{L1S2cpt} &  \|{\Phi}^\xi (f) \|_{L^{1,\infty}(Z,\sigma,\mathsf{s}_2)} \lesssim K^2    A_{3}({\Phi}^\xi )\|f\|_1
\end{align}
under the above compact support assumption. Let us detail how to derive \eqref{LpS2}. We apply the decomposition of Lemma \ref{lemmasplititer} for $Q=6$ say, and, referring to the notation therein, obtain the decomposition
\begin{equation} \label{spliteq}
{\Phi}^\xi (f)= \sum_{k\geq 0}2^{-6k}{\Psi}^{\xi;k}(f)
\end{equation}
Applying the  triangle inequality \eqref{trianin}, noting that $\supp \psi_{t;k} \subset[-2^{k}t,2^k t]$, and using \eqref{LpS2cpt}
\[
\begin{split}
\|{\Phi}^\xi (f) \|_{L^p(Z,\sigma,\mathsf{s}_2)} &\lesssim \sum_{k\geq 0}2^{-4k} \|{\Psi}^{\xi;k}(f) \|_{L^p(Z,\sigma,\mathsf{s}_2)}\lesssim     \|f\|_p\sum_{k\geq 0}2^{-2k}  A_{3}({\Psi}^{\xi;k} ) \\ & \lesssim  A_{C}({\Phi}^{\xi} )  \|f\|_p 
\end{split}
\]
which is \eqref{LpS2}.
\end{remark}
\begin{remark} \label{xizeroredux} Given a $\xi$-adapted system [resp.\ with mean zero] ${\Phi}^\xi=\{\phi_t\}$, the system ${\Phi}:= \{\e^{-i\xi\cdot} \phi_t(\cdot)\}$ is zero-adapted with same adaptation constants. Since
$$
|{\Phi}^\xi (f)|= |{\Phi} (f(\cdot) \e^{i\xi\cdot})|,
$$
we can reduce to the case $\xi=0$ when proving \eqref{LpS2cpt}-\eqref{L1S2cpt}. This is merely for notational convenience.
\end{remark}
\begin{proof}[Proof of \eqref{LpS2cpt}-\eqref{L1S2cpt}] 
We begin by proving the case $p=\infty$ of \eqref{LpS2cpt}. Let $I$ be any interval. We need to prove that
$$
\mathsf{s}_2({\Phi}(f)) (\T(I))^2= \frac{1}{|I|} \int_{\T(I)} |{\Phi}(f)(u,t)|^2 \, \frac{\d u \d t}{t}  \lesssim K^4 A_{3}({\Phi})^2 \| f\|_\infty^2.
$$
Due to  $\supp \phi_t\subset [-Kt,Kt]$, ${\Phi}(f{1}_{\R\setminus 3KI})(u,t)=0$ whenever $(u,t) \in \T(I)$. Thus using Lemma \ref{areaintegral},
\[
\begin{split}
\int_{\T(I)} |{\Phi}(f)(u,t)|^2 \, \frac{\d u \d t}{t} & = \int_{\T(I)} |{\Phi}(f{1}_{  3KI})(u,t)|^2 \, \frac{\d u \d t}{t} \leq \int_{Z} |{\Phi}(f{1}_{  3KI})(u,t)|^2 \, \frac{\d u \d t}{t} \\ & \lesssim A_{3}({\Phi})^2 \|f{1}_{  3KI}\|_2^2 
\lesssim  A_{3}({\Phi})^2 K |I| \| f\|_\infty^2,\end{split}
\]
which is (more than) what we sought.

Next we will prove \eqref{L1S2cpt}, and the remaining cases of \eqref{LpS2cpt} will follow from outer interpolation. We can assume, by linearity, that $\|f\|_1=1$. Given $\lambda>0$, let $I\in \cic{I}$ be the collection of  maximal dyadic intervals such that $I$ is contained in $\{\mathrm{M}f>\lambda\}$.  By the maximal theorem, the set $E= \bigcup_{I \in \cic I} \T(3I)$ has outer measure $\mu(E) \lesssim \lambda^{-1}$. It will then suffice to show that
\begin{equation}
\label{L1S2cpt0}
\sup_J
\mathsf{s}_2({\Phi}(f) \cic{1}_{Z\setminus E}) (\T(J)) \lesssim  A_{3}({\Phi}) K^2 \lambda
\end{equation}
Let $$
f=g+b= g+ \sum_{I \in \cic I} b_I
$$
be the Calder\'on-Zygmund decomposition of $f$ at level $\lambda$. Since $\|g\|_\infty\lesssim \lambda$, by the previous part of the proof we get
\begin{equation}
\label{L1S2cpt1}
\sup_J
\mathsf{s}_2({\Phi}(g) \cic{1}_{Z\setminus E}) (\T(J))  \leq \sup_J
\mathsf{s}_2({\Phi}(g) (\T(J)) \lesssim A_{3}({\Phi}) K \lambda.
\end{equation}
We now turn to bounding the contribution of $b$. Let $(u,t) \not \in \T(3I)$. By virtue of the support condition on $\phi_t$, 
$b_I * \phi_t(u)=0$ unless $t\geq K^{-1} |I|$. Letting $B_I$ denote the compactly, disjointly  supported primitives of $b_I$, we recall that
$$
\|B_I\|_\infty \lesssim \lambda |I|.
$$ Also, letting $\psi_t= (\phi_t)'$, we integrate by parts to obtain
\begin{equation}
\label{badpart1}|b\ast \phi_t(u)|= \left| \sum_{Kt\geq |I|} B_I \ast \psi_t(u)  \right|\leq  \left\|\sum_{Kt\geq |I|}t^{-1} B_I\right\|_\infty  \|t\psi_t\|_1  \lesssim A_{3}({\Phi}) K\lambda, \qquad (u,t) \not \in E.
\end{equation}
Thus
\begin{equation}
\label{badpartinfty}\sup_{J}\mathsf{s}_\infty({\Phi}(b) \cic{1}_{Z\setminus E}) (\T(J)) \lesssim KA_{3}({\Phi})  \lambda. \end{equation} Fix now $I \in \cic {I}$. Notice that ${\Phi}(b_I)=0$ on $\T(J)$ unless $I\subset 3KJ$. By the same support considerations and integrating by parts
\[
\begin{split} 
\int_{\T(J)\setminus E} |b_I\ast \phi_t(u)| \frac{\d u \d t}{t}&\leq  \int_{Kt\geq |I|} \int_{|u-c(I)|\leq 2Kt}\|B_I\|_1 \|\psi_t\|_\infty \frac{\d u \d t}{t} \\ &\lesssim K  A_{3}({\Phi})\lambda |I|^2  \int_{Kt\geq |I|} \frac{\d t}{t^2} \lesssim K^2  A_{3}({\Phi})\lambda |I|,
\end{split} 
\]
so that, summing over $I\subset 3KJ$ (which are pairwise disjoint, we obtain 
$$\mathsf{s}_1({\Phi}(b) \cic{1}_{Z\setminus E}) (\T(J)) \lesssim K^3  A_{3}({\Phi})\lambda. 
$$
Interpolating the above bound with \eqref{badpartinfty} it follows that 
\begin{equation}
\label{L1S2cpt2}
\sup_J
\mathsf{s}_2({\Phi}(b) \cic{1}_{Z\setminus E}) (\T(J))    \lesssim A_{3}({\Phi}) K^2 \lambda,
\end{equation}
which, combined with \eqref{L1S2cpt1}, completes the proof of \eqref{L1S2cpt0}, and, in turn, of \eqref{L1S2cpt}.
\end{proof}
\subsection{Localized embeddings} \label{ssloc}
 Let $\cic{I}$ be  a collection of intervals. The tent over $\cic{I}$ is defined as
\begin{equation}
\label{excset0}
E_{\cic{I}}= \bigcup_{I \in \cic{I}} \T(3I).\end{equation}
For technical reasons, we need to define $E_{\cic{I}}$ using the slightly enlarged tents $\T(3I)$. We warn the reader that this creates a discrepancy between the tent over an interval $I$, namely $\T(I)$,  and the tent over the family $\cic{I}=\{I\}$, which is $E_{\cic{I}}=\T(3I)$.
The next three results are  $L^\infty$ Carleson embeddings   outside or inside  $E_{\cic{I}}$-type sets.    
 \begin{lemma}
\label{loc1}
Let ${\Phi}^\xi $ be a $\xi$-adapted system. For   $f\in \mathcal S(\R)$, define
\begin{equation}
\label{excset}
  \cic{I}_{f,\lambda,1}:=\big\{\textrm{\emph{maximal dyadic intervals}  } I\subset \{x\in \R: \mathrm{M}f(x)>\lambda\}\big\}. 
\end{equation}
and, referring to \eqref{excset0}, write $E_{f,\lambda,1}:=E_{\cic{I}_{f,\lambda,1}}$. Then 
\[
\left\|{\Phi}^\xi (f) \cic{1}_{Z\setminus E_{f,\lambda,1}} \right\|_{L^\infty(Z,\sigma,\mathsf{s}_\infty)} \lesssim A_{C}({\Phi}^\xi )  \lambda.
\]
If furthermore ${\Phi}^\xi $ has mean zero
\begin{equation}
\label{loc12}\left\|{\Phi}^\xi (f) \cic{1}_{Z\setminus E_{f,\lambda,1}} \right\|_{L^\infty(Z,\sigma,\mathsf{s}_2)} \lesssim A_{C}({\Phi}^\xi )  \lambda. 
\end{equation}
\end{lemma}
\begin{proof} Perusal of the proof of \eqref{L1S2cpt0}, relying on the Calder\'on-Zygmund decomposition. The extension to the non-compactly supported case is then obtained via the same decomposition \eqref{spliteq} recalled in Remark \ref{remcptsupp}. 
\end{proof}
\begin{lemma}
\label{loc2}
Let ${\Phi}^\xi $ be a $\xi$-adapted system, $K\geq 1,$ and $Q$  be a positive integer. Let $\cic{J}$ be a finitely overlapping collection of intervals and $f\in \mathcal S(\R)$ be such that  \[
\label{suppfdoesnot}
\supp f \cap 27KJ = \emptyset \qquad \forall J \in \cic{J}.
\]
Then, referring to \eqref{excset0} for $E_{\cic{J}}$,
\begin{equation}
\label{loc21}
\left\|{\Phi}^\xi (f) \cic{1}_{  E_{\cic{J}}} \right\|_{L^\infty(Z,\sigma,\mathsf{s}_\infty)}\leq  C_{Q} A_{CQ}({\Phi}^\xi)  K^{-Q} \sup_{J \in \cic{J}} \inf_{x \in  J} \mathrm{M}f(x),
\end{equation}
and furthermore, if ${\Phi}^\xi $ has mean zero,
\begin{equation}
\label{loc22}
  \left\|{\Phi}^\xi (f) \cic{1}_{  E_{\cic{J}}} \right\|_{L^\infty(Z,\sigma,\mathsf{s}_2)} \leq C_Q A_{CQ}({\Phi}^\xi)   K^{-Q} \sup_{J \in \cic{J}} \inf_{x \in  J} \mathrm{M}f(x),
\end{equation}
\end{lemma}
\begin{proof} We prove \eqref{loc22}, the proof of \eqref{loc21} being exactly the same. 
First of all, we notice that 
\begin{equation} \label{loc221}
\left\|{\Phi}^\xi (f) \cic{1}_{  E_{\cic{J}}} \right\|_{L^\infty(Z,\sigma,\mathsf{s}_2)}=   \sup_{L \subset \R} \mathsf{s}_2({\Phi}^\xi(f) \cic{1}_{  E_{\cic{J}}}) (\T(L))  \leq C \sup_{J \in \cic{J}} \sup_{L \subset 9J} \mathsf{s}_2({\Phi}^\xi(f)) (\T(L)).\end{equation}
This is because if  $\T(L) $   intersects $ E_{\cic{J}} $ nontrivially, and there is no $J\in \cic{J}$ with $L\subset 9J $, it must be that $3J \subset  3L$ whenever $J\in\cic{J} $ is such that $ 3J  \cap L\neq \emptyset. $ Therefore 
$$
E_{\cic{J}} \cap \mathsf{T}(L) \subset \bigcup_{J \in \cic{J}: J\subset 3L} \mathsf{T}(3J),  
$$
whence
$$
\left(\mathsf{s}_2({\Phi}^\xi(f) \cic{1}_{  E_{\cic{J}}}) (\T(L))\right)^2 \leq \frac{1}{|L|} \sum_{J \in \cic{J}: J\subset 3L } |3J| \left(\mathsf{s}_2({\Phi}^\xi(f) ) (\T(3J))\right)^2 
$$
 which, by finite overlap of $J \in \cic{J}$  is less than  $C$ times the supremum on the right hand side of \eqref{loc221}. 
We will now bound the right hand side of \eqref{loc221}.  Fix $J \in \cic{J}$ and $L\subset 9J$, and let $$
\lambda= 9  \inf_{x \in  J} \mathrm{M}f(x).$$ Construct  $\cic{I}_{f,\lambda,1}$ and $E_{f,\lambda,1}$ as in \eqref{excset}. Since $27J\cap \supp f=\emptyset$, it is easy to see that $$
9J\cap 3H =\emptyset  \qquad \textrm{for all intervals } H \textrm{ with } \int_H |f| \, \d x >\lambda |H|,$$
which in turn   implies  \begin{equation}
\label{Tellcap}  \T(L)\cap E_{f,\lambda,1}=\emptyset.
\end{equation}
 From Lemma \ref{lemmasplit}, there  exists two   $\xi$-adapted systems  with mean zero ${\Psi}^{\xi}=\{\psi_t\}, {\Upsilon}^{\xi}=\{\upsilon_t\}$, such that for all $t$
$$
\phi_t = \psi_t + K^{-Q} \upsilon_t, \qquad \supp \psi_t 
\subset [-Kt,Kt].
$$
 Since $\dist (L,\supp f) \geq 9K |J|\geq K|L|$, the support of $(u,t)\mapsto {\Psi}^\xi(f)(u,t) $ does not intersect $T(L)$. Therefore
\[
{\Phi}^\xi(f) (u,t) = K^{-Q} {\Upsilon}^{\xi}(f) (u,t) = K^{-Q} {\Upsilon}^{\xi}(f) \cic{1}_{Z\setminus E_{f,\lambda,1}} (u,t), \qquad (u,t) \in \T(L),
\]
where the second equality is a consequence of \eqref{Tellcap}. By \eqref{loc12}, 
$$
  \mathsf{s}_2({\Phi}^\xi(f)) (\T(L)) \leq K^{-Q}
 \left\|{\Upsilon}^\xi(f)   \cic{1}_{Z\setminus E_{f,\lambda,1}} \right\|_{L^\infty(Z,\sigma,\mathsf{s}_2)} \lesssim  K^{-Q} A_{C}({\Upsilon}^\xi) \lambda \leq C_Q A_{CQ}({\Phi}^\xi) K^{-Q} \lambda,
$$
relying on \eqref{lemmasplititer2} to get the last equality. This proves the claimed bound of \eqref{loc221}.\end{proof}

The final lemma of this section is a strengthening of Lemma \ref{loc1}  in the case where the function involved    is a sum of Calder\'on-Zygmund atoms with $  Q$ vanishing moments. This produces an exponential gain in  $Q$.
\begin{lemma} \label{loc3}  Let $\xi \in \R$  and let $\cic{L}$
be a countable collection of  intervals. 
Let $b\in L^1_{\mathrm{loc}}(\R)$ be such that 
 \begin{align} & b = \sum_{L \in \cic{L}} b_L, \qquad \label{finovBMO}  \supp b_L \subset L ,  \qquad \sup_{I \subset \R} \frac{1}{|I|} \sum_{L \subset I} \|b_L\|_1 =:
\lambda<\infty
  \\ &  \label{vanishing1}\int b_L(x)  x^j \e^{ - i \xi x } \, \d x =0 \qquad \forall j=0,\ldots, 2(Q+1).\end{align}
 Let  $K>3$  be given and denote by $K\cic{L}=\{KL: L \in \cic{L}\}$.
Let ${\Phi}^{\xi}$ be a  $\xi$-adapted system [resp.\ with mean zero]. Then, for  $p =\infty$ [resp.  $p=2$],
\begin{equation}
\label{estsize}
\left\|{\Phi^\xi}(b) \cic{1}_{Z\setminus {E_{K\cic{L}}}}\right\|_{L^\infty(Z,\sigma,\mathsf{s}_p)}
\leq   C_{\cic{L}}  C_Q A_{CQ}({\Phi}^\xi ) K^{-Q} \lambda.
\end{equation}
 \end{lemma}
 \begin{remark}\label{suffcond} Note that \eqref{finovBMO}, with $C\lambda$ in place of $\lambda$, follows if it is known that 
 \begin{equation}
\label{finov}  \supp b_L \subset L, \qquad 
 \sup_{L\in \cic{L}}\frac{1}{|L|}\|b_L\|_1 \leq \lambda, \qquad  
\sup_{I\in \R} \frac{1}{|I|} \sum_{\substack{L \in \cic{L}\\ L \subset I}}|L| \leq C.
\end{equation}
 \end{remark}
 \begin{remark} \label{themftrick} Before the actual proof, we   record that, as a consequence of \eqref{finovBMO}, there holds
 \[
I \not \subset \widetilde E:=\bigcup_{L \in \cic{L} } 3L \implies  \|b\cic{1}_I\|_1 \leq      \sum_{\substack{L \in \cic{L}\\ L \subset 3I}}   \|b_L\|_1 \leq 3 \lambda  |I|
   \]
   which in turn implies that
 $
    \{x \in \R: \mathrm{M} b(x)>3 \lambda \}\subset \widetilde E
   $. Namely, referring to the notations below \eqref{excset} and in \eqref{excset0}, is that if $I \in  \cic{I}_{b, 3 \lambda, 1}$ then $3I$ is covered by the intervals $\{9L: L\in \cic{L}\}$ whence
   \begin{equation}
\label{themftrick2}
E_{b, 3 \lambda, 1} \subset E_{ 3\cic{L} }  \subset E_{ K \cic{L} } .
\end{equation}
 \end{remark}
 \begin{proof}[Proof of Lemma \ref{loc3}] We prove the mean zero case. The proof of the other one is actually within this case. To lessen the notational burden, one may reduce to the case $\xi=0$ by the same argument as in Remark \ref{xizeroredux} (i.\ e.\ replacing $b(\cdot)$ by $b(\cdot) \e^{i\xi \cdot}$), and simply call  ${\Phi}$ the 0-adapted system with mean zero.

 We begin the actual proof. Clearly, for \eqref{estsize} it suffices to estimate $\mathsf{s}_2({{\Phi}}(b) \cic{1}_{Z\setminus E_{ K \cic{L} }})(\mathsf{T}(J))$ for  those intervals $J$ with  $\mathsf T(J) \not \subset   E_{ K \cic{L} }$. 
Fix one such $J$ from now on and let $\kappa=\sqrt{K}$.
We   apply   Lemma \ref{lemmasplit} with $\kappa$ in place of $K$ and $2Q$ replacing $Q$, to find   adapted systems  with mean zero ${\Psi} =\{\psi_t\}, {\Upsilon}=\{\upsilon_t\}$, such that
\begin{equation} \label{loc32}
\phi_t = \psi_t + K^{-Q} \upsilon_t, \qquad \supp \psi_t 
\subset [-\kappa t,\kappa t].
\end{equation}
Using \eqref{themftrick2} for the first inequality, we apply the  estimate \eqref{loc12} from Lemma \ref{loc1}, and subsequently  \eqref{lemmasplititer2}, yielding \begin{equation} \label{loc33}
 \mathsf{s}_{2}\left({{\Upsilon}}(b) \cic{1}_{Z\setminus E_{K\cic{L}}} \right) (\mathsf{T}(J))\leq   \mathsf{s}_{2}\left({{\Upsilon}}(b) \cic{1}_{Z\setminus E_{b, 3\lambda, 1}} \right) (\mathsf{T}(J)) \lesssim  A_{C}({\Upsilon} )\lambda   \leq   C_Q A_{CQ}({\Phi}^\xi )  \lambda.   
\end{equation}
The treatment of ${\Psi}(b) $ is very similar to the proof of \eqref{L1S2cpt2}, the difference being that  $Q$ integration by parts are performed. 
We make two important observations. The first is that, in view of the support of $\psi_t$ being as in \eqref{loc32},
\begin{equation}
\label{loc34} (u,t) \not \in \T(3KL) \implies {\Psi}(b_L)(u,t) = 0 \textrm{ unless } t\geq \kappa |L|, \, L \subset (u-3\kappa t,  u+3\kappa t)
\end{equation}
The second, by the same reasons, is that 
\begin{equation}
\label{loc35} (u,t) \in \T(J) \implies {\Psi}(b_L)(u,t) = 0 \textrm{ unless } L\subset 9\kappa J
\end{equation}
Iteratively define
$$
b_{L,0}:= b_L, \quad b_{L,j+1}:= \int b_{L,j} (x) \,\d x,\quad j=0,\ldots 2Q+1, \qquad B_L= b_{L,2(Q+1)}.
$$
The    moment conditions \eqref{vanishing1} ensure that  $B_L$ are   supported on $L$ and 
\begin{equation}
\label{loc36}
\|B_L\|_\infty \leq   |L|^{2Q+1}\|b_L\|_1.
\end{equation}
Let now $\varphi_t= (\psi_t)^{(2Q+2)}$. Similarly to \eqref{badpart1},  we  integrate by parts $2(Q+1)$ times to obtain
\[\begin{split}
|{\Psi}(b)(u,t)|&\leq   \sum_{ \substack{ t \geq \kappa |L|\\ L \subset (u-3\kappa t,  u+3\kappa t)}} \left|B_L \ast \varphi_t(u)  \right|\leq  \sum_{ \substack{ t \geq \kappa |L|\\ L \subset (u-3\kappa t,  u+3\kappa t)}}t^{-2(Q+1)} \left\|B_L\right\|_\infty  \|t^{2(Q+1)}\varphi_t\|_1 \\& \leq \kappa^{-2Q} (\kappa t)^{-1} \left(\sum_{ \substack{   L \subset (u-3\kappa t,  u+3\kappa t)}} \|b_L\|_1  \right)  A_{2(Q+2)}({\Psi})   \lesssim A_{2(Q+2)}({\Psi}) \kappa ^{-2Q} \lambda ,
\end{split}
\]
whenever  $ (u,t) \not \in {E_{K \cic{L}}}.$
Note that the restriction in the first sum is due to \eqref{loc34} and we have used it, together with \eqref{loc36}, to pull out a $\kappa^{-2Q-1}$ factor. The last inequality is obtained thanks to  \eqref{finovBMO}.
Therefore \begin{equation}
\label{loc37}
\mathsf{s}_\infty({\Phi}(b) \cic{1}_{Z\setminus E_{K\cic{L}}}) (\T(J)) \lesssim     A_{2(Q+2)}({\Psi}) \kappa ^{-2Q}\lambda \leq  C_Q A_{CQ}({\Phi}  ) K^{-Q} \lambda.\end{equation}
Note that \eqref{loc36} entails $\|B_L\|_1\leq |L|^{2Q+2}\|b_L\|_1$. Making use of the  observation \eqref{loc35}, another  integration by parts yields
\[
\begin{split} 
&\quad \int_{\T(J)\setminus {E_{K \cic{L}}}} |b_L\ast \psi_t(u)| \frac{\d u \d t}{t}\leq  \int_{t\geq \kappa |L|} \int_{|u-c(L)|\leq \kappa t}  \|B_L\|_1 \|t^{2Q+3}\varphi_t\|_\infty \frac{\d u \d t}{t^{2Q+4}} \\ &\lesssim  A_{2(Q+2)}({\Psi})  |L|^{2Q+2}\|b_L\|_1  \int_{t\geq \kappa |L|} \frac{\kappa \d t}{t^{2Q+3}} \lesssim      A_{2(Q+2)}({\Psi}) \kappa^{-2Q-1}  \|b_L\|_1.
\end{split} 
\]
Dividing by $|J|$, summing over   $L\subset 9\kappa J$, and using again the assumption \eqref{finovBMO}  we obtain, also in view of \eqref{lemmasplititer2}, the inequality
\begin{equation}
\label{loc38}\mathsf{s}_1({\Psi}(b) \cic{1}_{Z\setminus {E_{K \cic{L}}}}) (\T(J)) \leq  C_Q A_{CQ}({\Phi}^\xi ) K^{-Q} \lambda
\end{equation}
An interpolation between \eqref{loc38} and \eqref{loc37} yields
\begin{equation}
\label{loc39}
\mathsf{s}_2({\Psi}(b) \cic{1}_{Z\setminus E}) (\T(J))  \leq  C_Q A_{CQ}({\Phi}^\xi ) K^{-Q} \lambda
\end{equation}
and combining \eqref{loc39} with \eqref{loc33} and the decomposition \eqref{loc32} finishes the proof of the claimed estimate.
 \end{proof}
 
 \begin{remark} Lemmata \ref{loc1} and \ref{loc2} are outer measure analogues of two usual time-fre\-quen\-cy analysis lemmata, the John-Nirenberg inequality (see for example \cite[Proposition 2.4.1]{ThWp}) and the localization trick, used for instance to estimate the contribution of the part of the model operator localized inside the exceptional set.  Lemma \ref{loc3}, on the other hand, has no close discrete analogue. Its purpose is to control the bad part arising in the multi-frequency CZ decomposition which will be used in the proof of Theorem \ref{bademb}.
 \end{remark}

 \section{Proofs of Theorem \ref{bademb} and Proposition \ref{goodemb}} \label{secgentents}

  In the upcoming subsection, we prove Proposition \ref{goodemb}, which will follow directly from the embeddings of Section \ref{secembzero}. The remainder of the section is devoted to the proof of Theorem \ref{bademb}. Subsection \ref{ss72} contains the preliminary tools leading to the main argument in Subsection \ref{pfcet}, which in turn reduces \eqref{cet} from Theorem \ref{bademb} to the multi-frequency Calder\'on-Zygmund  Lemma \ref{lemmaCZ}. The proof of Lemma \ref{lemmaCZ} is postponed to Section \ref{sec8}.

\subsection{Proof of Proposition \ref{goodemb}} 
We begin with a remark linking outer measure spaces on classical tents with their  generalized  counterpart. To highlight the dependence of this observation on the geometric parameters $\cic{\alpha}$, we go back to explicitly writing the subscript $\cic{\alpha}$ till the end of the subsection. Let $F\in \mathcal B(\mathcal Z)$ and define 
$$
 F_{\xi,\theta}\in \mathcal B(Z), \qquad  F_{\xi,\theta}(u,t):= F\left(u,t,\xi+\textstyle\frac{\theta-\beta}{\alpha t}\right).
$$  Comparing  the definitions of \eqref{size} and \eqref{size0},
a change of variables entails
\[
\begin{split}&
\mathsf{s}_2(F) (\T_{\cic{\alpha}}(I,\xi))^2=\alpha^{-1}\int_{\Theta_{\cic{\alpha}}} \Big( \mathsf{s}_2\left(F_{\xi,\theta}
\right)(\T_{\cic{\alpha}}(I))\Big)^2 \, \d \theta, \qquad \Theta_{\cic \alpha}:=\{|\theta|<1, |\theta-\beta|>\alpha \delta\},  \\
& \mathsf{s}_\infty(F) (\T_{\cic{\alpha}}(I,\xi)) = \sup_{|\theta |< 1} \left(\mathsf{s}_\infty\left(F_{\xi,\theta}\right)(\T_{\cic{\alpha}}(I))\right) 
\end{split}
\]
in consequence of which, if $G\subset  Z $ is a Borel set  \begin{equation}
\label{sizeaver}
\| F \cic{1}_{G\times \R} \|_{L^\infty(\mathcal Z, \sigma_{\cic \alpha}, \mathsf{s})  } \leq\sup_{ \xi \in \R} \left( \sup_{|\theta| < 1 } \| F_{\xi,\theta} \cic{1}_G \|_{L^\infty( Z, \sigma_{\cic \alpha}, \mathsf{s}_\infty)  } +  \alpha^{-\frac12} \sup_{\theta \in \Theta_\alpha} \| F_{\xi,\theta} \cic{1}_G\|_{L^\infty( Z, \sigma_{\cic \alpha}, \mathsf{s}_2)  } \right).
\end{equation}
 In relation to \eqref{sizeaver}, and recalling that $F_\phi(f)(u,t,\eta)=f*\phi_{t,\eta}(u)$ as well as the notation \eqref{Phixi}, one has the equality
\begin{equation}
\label{dilmodgen1}
 \left(F_\phi(f)\right)_{\xi,\theta} = {\Phi}^{\xi;\theta}(f), \qquad \xi \in \R, \,|\theta| < 1
\end{equation}
where
\[
{\Phi}^{\xi;\theta}=\left\{ \phi_{t,\eta(\theta)}: t \in (0,\infty)\right\}, \qquad\eta(\theta):= \xi+ t^{-1}\frac{\theta-\beta }{\alpha}
\]
are $\xi$-adapted systems with adaptation    
constants 
$$
A_{N}({\Phi}^{\xi;\theta}) \leq C_N\alpha^{-CN} \mathcal{S}_{N}(\phi).
$$ In particular, these adaptation constants are uniform in $\theta \in (-1,1)$. If one further restricts to $\Theta_{\alpha}$, we obtain a $\xi$-adapted system  with mean zero as well. 

We will now use these observations in the proof of Proposition \ref{goodemb}. Assume $f$ and $\cic{J}$ satisfy \eqref{goodemb3a} for some  $K>1$. It suffices, by virtue of \eqref{goodemb1}, to argue for large $K$.  In view of    \eqref{sizeaver} and \eqref{dilmodgen1}, and  the uniform adaptedness of ${\Phi}^{\xi;\theta}$,  an application of the estimates  \eqref{loc21} and \eqref{loc22} from Lemma \ref{loc2} for a slightly smaller  $K$ and for some $Q>1$,  entail
\begin{equation}
\label{goodemb3b}
\|F_{\phi } (f)\cic{1}_{E_{\cic{J}}}\|_{L^{\infty}( \mathcal Z, \sigma_{\cic{\alpha}},\mathsf{s} )} \leq  C_{{\alpha},\phi,Q} K^{-Q}  \sup_{J \in \cic{J}} \inf_{x \in  J} \mathrm{M}f(x) .\end{equation}
The supremum on the right hand side is obviously bounded by $\|f\|_\infty$, whence  \eqref{goodemb3} for $q=\infty$, $Q=R$.   To obtain the case $2<q<\infty$, we use outer Marcinkiewicz interpolation between \eqref{goodemb3b}, with $Q$ suitably chosen depending on $R$ and $q$, and \eqref{goodemb2}, on the linear operator
$$
f \mapsto F_\phi\left(f \cic{1}_{(\bigcup_{J \in \cic {J}} KJ)^c}\right) \cic{1}_{E_{\cic{J}}}.
$$
This completes the proof of Proposition \ref{goodemb}. 
 
\subsection{Proof of Theorem \ref{bademb}: preliminaries} \label{ss72} First of all, we record two lemmata which are obtained respectively from Lemma \ref{loc1} and \ref{loc3} via  an argument analogous to the one used for Proposition \ref{goodemb} and involving \eqref{sizeaver}. The first is the $q=\infty$ easy endpoint of \eqref{cet} from Theorem \ref{bademb}. The second is  key to the estimation of the \emph{bad part} leading to  the embedding \ref{cet}.
\begin{lemma} \label{lemmabademb1} Let $f\in \mathcal S(\R)$, $\lambda>0 $ and $1\leq p<  2$, and refer to the notation of \eqref{maxfctbd} for $E_{f,\lambda,p}$. Then \begin{align}  & \label{bademb1}
\left\|F_\phi (f) \cic{1}_{\mathcal Z\setminus E_{f,\lambda, p}} \right\|_{L^\infty(\mathcal Z, \sigma_{\cic{\alpha}},\mathsf{s})} \leq C_{\alpha,\phi}{\lambda\|f\|_p} 
\end{align}
 with $C_{\alpha,\phi}=C\alpha^{-C} \mathcal{S}_C(\phi).$ 
\end{lemma}
\begin{lemma} \label{gloc3}  Let $\xi \in \R$,  and
$b$, $\cic{L}$, and 
 $\lambda>0$ be as in Lemma \ref{loc3}. 
 Let $K>1$ be a given constant and denote
\begin{equation}
\label{gconstraint}
E_{K\cic{L}} := \bigcup_{L \in \cic{L}} \mathsf{T} (3KL) \times \R \subset \mathcal Z.
\end{equation}
Then 
\begin{equation}
\label{gestsize}
\sup_{J\subset \R}  \mathsf{s} \left(F_\phi( b) \cic{1}_{\mathcal Z\setminus E} \right) (\mathsf{T}(J,\xi)) \leq C_{\alpha,\phi,Q} K^{-Q} \lambda,
\end{equation}
with $C_{\alpha,\phi,Q}=C_Q\alpha^{-CQ}   \mathcal{S}_{CQ}(\phi)$.
 \end{lemma}
We will  also need     several  components of the  proof of \eqref{goodemb2}, which in fact can be reconstructed by combining the definitions and lemmata that follow. Since $\cic \alpha$ will be thought of as fixed throughout this paragraph, as well as for  the next section, we  omit it from the notation and write $\T$ in place of $\mathsf{T}_{\cic \alpha}$, $\sigma $ in place of $\sigma_{\cic \alpha}$, $\nu$ in place of $\nu_{\cic \alpha}$.
In the upcoming definitions, the   constant $C_{{ \alpha},\phi}$ is allowed to depend on the parameter $ {\alpha}$, more precisely, polynomially in $|\alpha|^{-1}$,  and on $\mathcal S_C(\phi)$ for large enough $C$.

\begin{remark}[Reduction to tents with discrete parameters]\label{remcet1} Let $\mathcal D$ be a finite union of dyadic grids on $\R$. We denote by $\cic{E}_{\mathcal D}$ the collection of generalized tents $\mathsf{T}(I,\xi)$ such that $I \in \mathcal D$ and $\xi \in \delta|I|^{-1} \mathbb Z$ for some small, fixed  dyadic parameter $\delta$. Momentarily, denote by $\mu_{\mathcal D}$ the outer measure generated by the collection $\cic{E}_{\mathcal D}$ via the premeasure $\sigma_{\mathcal D}(\mathsf{T}(I,\xi))=|I|$.  Following the reduction presented in \cite[Lemma 5.2]{DoThiele15}, there is a suitable choice of $\mathcal D$  yielding   equivalence of outer $L^p$ spaces
$$
L^p(\mathcal Z, \sigma_\mathcal{D}, \mathsf{s}) \sim L^p(\mathcal Z, \sigma, \mathsf{s}), \qquad L^{p,\infty}(\mathcal Z, \sigma_\mathcal{D}, \mathsf{s}) \sim L^{p,\infty}(\mathcal Z, \sigma, \mathsf{s})
$$
with absolute constant. 
Relying upon  the above equivalence, we   switch to working with the  spaces $L^p(\mathcal Z, \sigma_\mathcal{D}, \mathsf{s})$ and their weak counterparts in place of $L^p(\mathcal Z, \sigma, \mathsf{s})$. However, we drop the $\mathcal D$, thus reverting to the original notation. There is no loss in generality in actually assuming that $\mathcal D$ is the standard dyadic grid.
\end{remark}

\begin{definition} A countable collection  $\Tt=\{\T_n\}$ of generalized tents is said to be $\infty$-\emph{strongly disjoint} if the following holds: there exist points $(u_n,t_n,\eta_n)\in \T_n$ with $t_n\geq 2^{-3}|I_{\T_n}|$, such that for all $g\in L^2(\R)$, if
$$
\rho:=\inf_{n} F_\phi(g)(u_n,t_n,\eta_n) >0
$$ 
there holds
\begin{equation}
\label{infsd1}
\rho^2\sum_{\T_n \in \Tt} |I_{\T_n}|  \leq C_{\alpha,\phi} \|g\|_2^2.
\end{equation}
\end{definition}
\begin{definition} A countable collection  $\Tt$ of generalized tents is said to be $2$-\emph{strongly disjoint} if the following holds. For each $\T \in \Tt$ there exists a distinguished subset $\T^* \subset \T^\ell$ such that 
whenever $g\in L^2(\R)$ satisfies  
\begin{equation}\begin{split}
&
\sup_{\T \in \Tt } \sup_{(u,t,\eta) \in \T^*}   \left|F_\phi(g)(u,t,\eta)\right| \leq C \rho,
\\ & C^{-1} \rho   \leq\left( \frac{1}{|I_\T|}\int_{\T^*} |F_\phi(g)|^2 \, \d u \d t \d \eta\right)^{\frac12} \leq C \rho    \ \qquad \forall \T \in \Tt
\end{split}
\end{equation}
for some $\rho>0$,
there holds
\begin{equation}
\label{2sd1}
\rho^{2} \sum_{\T \in \Tt} |I_\T| \leq C
\sum_{\T \in \Tt} \int_{\T^*} |F_\phi(g)|^2 \, \d u \d t \d \eta \leq C_{\alpha,\phi}  \|g\|_2^2.
\end{equation}
\end{definition}
\begin{lemma}[Selection algorithm] \label{lemmasele} Let $f\in \mathcal S(\R)$ with compact frequency  support and $E\subset \mathcal Z$ be given.
Denote by $F:=F_\phi(f)\cic{1}_{E^c}$ and assume that for some $\lambda>0$
$$
 \sup_{\T} \mathsf s(F_\phi(f)\cic{1}_{E^c})(\T) >\lambda.
$$
 Then, there exist an $\infty$-strongly disjoint collection   $\Tt_0=\{\T_n\}$ with distinguished points
 $(u_n,t_n,\eta_n)\in \T_n,$  and a $2$-strongly disjoint collection  $\Tt_1=\{\T\}$ with distinguished subsets $\T^*\subset \T^\ell$ such that
\begin{align}
& \nu\left( \mathsf{s}(F )>\lambda\right)\leq \sum_{\T \in \Tt_0 \cup \Tt_1} |I_\T|,\label{lemmasele1}\\
& \inf_{\T_n \in \Tt_0} |F (u_n,t_n,\eta_n)| > \lambda ,\label{lemmasele2}
\\ & \sup_{\T \in \Tt_1}\sup_{(u,t,\eta) \in \T^*}   \left|F (u,t,\eta)\right| \leq \lambda \label{lemmasele3} ,
\\ &     2^{-4} \lambda \leq \left(\frac{1}{|I_\T|}  \int_{\T^*} |F  |^2 \, \d u \d t \d \eta \right)^{\frac12}\leq \lambda \qquad \forall \T \in \Tt_1.\label{lemmasele4}
\end{align}
\end{lemma}
\begin{proof} We will carry on explicitly the construction of $\Tt_0$, and prove that it is an $\infty$-strongly disjoint collection, by perusing the first part of  the argument of \cite[Subsection 5.2]{DoThiele15}, concerning the $\mathsf s_\infty$ part of $\mathsf s$. Proceeding as in \cite{DoThiele15}, we may select a  (finite or countably infinite) collection of tents $\Tt_0:=\{\T_n\}$, whose union we denote by $H$, and   $(u_n,t_n,\eta_n)\in\T_n\cap E^c$  such that
$$
\inf_{n}F_\phi(f)(u_n,t_n,\eta_n) > \lambda, \qquad  \sup_{\T} \mathsf s_\infty(F\cic{1}_{H^c})(\T) \leq \lambda.
$$
The first property of the last display is  \eqref{lemmasele2}. The second property ensures  that 
\begin{equation} \label{lemmasele11}
\nu\left( \mathsf{s}_\infty(F )>\lambda\right)\leq \sum_{\T \in \Tt_0 } |I_\T|.
\end{equation}
The construction of the sequence $\{(u_n,t_n,\eta_n)\}$  ensures that, setting $
\phi_n(\cdot)= \sqrt{t_n} \phi_{t_n,\eta_n}(u_n -\cdot)
$, there holds $$
\sup_{n}\sum_{k: t_k \leq t_n}\left(\frac{t_k}{t_n}\right)^{\frac12}|\l \phi_k, \phi_n\r| \leq C_{\alpha,\phi};
$$
this is \cite[eq.\ (5.11)]{DoThiele15}.
Now, recalling that for any $g \in L^2(\R)$, one has $ \sqrt{t_n} F_\phi(g)(u_n,t_n,\eta_n) = \l g, \overline \phi_k \r $ the exact same argument of  \cite{DoThiele15} yields, for all $m\geq 0$,
$$
\sum_{n :\rho2^{m}\leq F_\phi(g)(u_n,t_n,\eta_n) <\rho2^{m+1}} \rho^2 |I_{\T_n}|\leq C     2^{-2m}\sum_{n :\rho2^{m}\leq F_\phi(g)(u_n,t_n,\eta_n) <\rho2^{m+1}}   |\l g, \overline \phi_k \r|^2 \leq C_{\alpha,\phi}2^{-2m}\|g\|_2^2
$$
and the claim \eqref{infsd1} follows by summing over $m$. We have thus  proved that $\Tt_0$ is a $\infty$-strongly disjoint collection.
At this point, a (finite or countable) collection  $\T\in \Tt_1$ of tents with distingushed subsets $\T^*\subset \T^\ell$ satisfying \eqref{lemmasele3}-\eqref{lemmasele4}, and such that
\begin{equation} \label{lemmasele12}
\nu\left( \mathsf{s}_2(F )>\lambda\right)\leq \sum_{\T \in \Tt_1 } |I_\T|,
\end{equation}
 can be 
constructed by using the \emph{second} part of \cite[Subsection 5.2]{DoThiele15}. Notice that we finally obtain the estimate \eqref{lemmasele1} by coupling \eqref{lemmasele11} with \eqref{lemmasele12}. The same ideas used for $\Tt_0$ will yield that $\Tt_1$ is a 2-strongly disjoint collection, and this completes the proof of the lemma.
\end{proof}
\subsection{Proof of  Theorem \ref{bademb}: main argument} 
\label{pfcet} Throughout this argument  $1<p<2$, $q>p'$ are fixed.  We set
\begin{equation}
\label{eps0}  \overline{\eps}= \overline{\eps}_{p,q}:=  \textstyle \frac{1}{p'} -\frac{1}{q}, \qquad {\eps}:=2^{-10}\overline{\eps}. \end{equation}
\subsubsection{Reduction to compact frequency support} We  show that Theorem \eqref{cet} can be reduced to the case of $f$ having compact frequency support, in which case the sets in \eqref{maxfctbd} can be replaced by  \begin{equation}  \label{maxfctbd2} \begin{split}  &
\underline{\cic{I}}_{f,\lambda,p},= \textrm{{max.\ dyad.\ int.\ }} I\textrm{\emph{ s.t. }}
 I\subset \left\{x\in\R: \mathrm{M}_p f(x)>{\lambda\|f\|_p}\right\},\\  & \underline{E}_{f,\lambda,p}=\bigcup_{I\in \underline{\cic{I}}_{f,\lambda,p}}  { \T}(3I) \times \R .\end{split}
 \end{equation}
  Note that, in particular, $\underline{E}_{f,\lambda,p}$ from \eqref{maxfctbd2} is contained in the corresponding version from \eqref{maxfctbd}. 
  
  Let $f \in \mathcal S(\R)$, $\lambda>0$ be fixed. Choose an increasing sequence of integers $\xi_k$ with $\xi_0=0$ and with the property that, setting  $ A_0=(-\xi_2,\xi_2)$, $A_k=\{\xi_{k}<|\xi|<\xi_{k+2}\}$ for $k \geq 1$, and $g_k:=\F^{-1}(\widehat f\cic{1}_{ A_k})$, there holds $$\left\|g_k\right\|_p\leq C_q 2^{-10qk} \|f\|_p, \qquad k\geq 0 $$
  Let $\psi_k$ be a smooth partition of unity subordinated to $\{A_k\}$,  and write 
  $
 f_k:=f * \widehat{\psi_k}. 
  $
  Note that $\|f_k\|_p\leq C\|g_k\|_p\leq C_q2^{-10qk} \|f\|_p$, and furthermore $|f_k|\leq \mathrm Mf$ pointwise. This yields, comparing the definitions \eqref{maxfctbd} and \eqref{maxfctbd2}, that
  $$
 \underline{E}_k:= \underline{E}_{f_k,C_q^{-1} 2^{10qk}\lambda,p} \subset  {E}_{f, \lambda,p}
  $$
This observation, coupled with \eqref{trianin}, and with the use of   \eqref{cet} for each $f_k$ which has compact support in frequency, yields\[
 \begin{split}
 \left\|F_{\phi } (f) \cic{1}_{\mathcal Z\setminus  E_{f,\lambda,p}}\right\|_{L^{q} (\mathcal Z, \sigma_{\cic{\alpha}},\mathsf{s})} &\leq   \sum_{k=0}^\infty 2^{2k}  \left\|F_{\phi } (f_k) \cic{1}_{\mathcal Z\setminus  \underline{E}_{k}}\right\|_{L^{q} (\mathcal Z, \sigma_{\cic{\alpha}},\mathsf{s})} \\ &\leq    C_{p,q,\alpha,\phi} \lambda^{1-\frac{p}{q}} \sum_{k=0}^\infty 2^{2k} 2^{10(q-p)k}\|f_k\|_p \leq C_{p,q,\alpha,\phi} \lambda^{1-\frac{p}{q}} \ \|f\|_p 
 \end{split}
 \]  
 completing the reduction.
 In the argument that follows, we work with the definitions \eqref{maxfctbd2} in place of \eqref{maxfctbd}. For convenience, we drop the underline and return to the notation of \eqref{maxfctbd}.
\subsubsection{Main line of proof} Let now $f\in \mathcal{S}(\R)$ be   fixed and have compact frequency support. By   vertical scaling we can assume $\|f\|_p=1$. To obtain estimate \eqref{cet}, it suffices to prove  the analogue   where   $L^{q } (\mathcal Z, \sigma_{\cic{\alpha}},\mathsf{s})$ is replaced by the corresponding outer $L^{q,\infty}$.
The upgrade to outer $L^{q }$ then follows from log-convexity of the  norms of $ F_{\phi } (f) \cic{1}_{\mathcal Z\setminus  E_{f,\lambda,p}}$, see \eqref{Lplog} from Proposition \ref{propoutermeassumm}. 

  Furthermore, by the dilation invariance \eqref{scalingest} described in Remark \ref{scalingrem}, it suffices to   work with a single value of $\lambda$. 
 We choose the value \begin{equation}
\label{A0}
  \lambda=\lambda_{p,q, {\alpha}, \phi}:= \left(2  C_{ {\alpha},\phi}\right)^{-1} 2^{-\frac{k_0}{q}}  \leq 1,
\end{equation}
where $C_{\alpha,\phi}$ is the constant appearing in \eqref{bademb1} and $2^{k_0}$  is $2^{10}C_0\eps^{-1}$ times  the greater of the two constants $ C_{{\alpha},\phi,Q}$  appearing in \eqref{goodemb3b} and \eqref{gestsize}, for the choice $Q=2^{5} {  \eps}^{-1}$. This way, $k_0=C_{p,q, {\alpha},\phi}$. Here,  $C_0$ is an absolute constant which will be specified later in the proof, see \eqref{techpf}.    
 We can use  the already established \eqref{bademb1} from Lemma \ref{lemmabademb1} and obtain $$
\|F_\phi (f) \|_{L^\infty(\mathcal Z\setminus E_{f,\lambda, p},\sigma_{\cic{\alpha}},\mathsf{s})} \leq  2^{-\frac{k_0}{q}}.
$$ 
The assertion of   \eqref{cet} can then be equivalently reformulated as follows: find $C_{p,q, \alpha,\phi}$ large enough such that
\begin{equation}
\label{ref1}
\sup_{k > k_0}
2^{-k}\nu \left(\mathsf{s} (F ) >    2^{-\frac{k}{q}}\right) \leq C_{p,q,  {\alpha},\phi}, \qquad F:= F_\phi(f) \cic{1}_{\mathcal{Z}\setminus E_{ f,\lambda,p}}.
\end{equation} 
 The collection of dyadic intervals 
\begin{equation}
\label{outdyadint}
\cic{J}= \left\{ J\in \mathcal D:  \inf_{x \in J} \mathrm{M}_p f(x) \leq \lambda
\right\}
\end{equation}
will play a role in our main line of proof, which begins now.

From now on, we   fix $k>k_0$.
 The first step is  the application of the selection Lemma \ref{lemmasele}, with $\lambda$ replaced by $2^{-\frac{k}q}$, yielding an $\infty$-strongly disjoint collection $\Tt_0$ and a $2$-strongly disjoint collection $\Tt_1$, with the properties  
\begin{align}
& \left\{ I_{\T}:\T \in \Tt\right\} \subset    \cic{J},  \label{outexcset}\\
& \label{cetlocpf1} \nu_{\cic{\alpha}}\left( \mathsf{s}(F)>2^{-\frac kq}\right)\leq \sum_{\T \in \Tt} |I_\T|=: 2^{N},\\
&  \inf_{\T \in \Tt_0}\mathsf{s}_\infty(F )(\T)> 2^{-\frac kq}, \label{lemmasele1app}
\\ & \sup_{\T \in \Tt_1}\sup_{(u,t,\eta) \in \T^*}   \left|F(u,t,\eta)\right| \leq 2^{-\frac kq}, \label{lemmasele2app}
\\ &    2^{-\frac kq-4} \leq \left(\frac{1}{|I_\T|}  \int_{\T^*} |F |^2 \, \d u \d t \d \eta \right)^{\frac12}\leq 2^{-\frac kq} \qquad \forall \T \in \Tt_1. \label{lemmasele3app}
\end{align} We wrote for simplicity $\Tt=\Tt_0\cup \Tt_1$. Property \eqref{outexcset} follows by construction of $E_{f,\lambda,p}$, see \eqref{maxfctbd}.  
It is clear that \eqref{ref1} follows immediately if we prove that \begin{equation}
\label{thefinalbd}
2^N\leq C_{  p,q, \alpha, \phi } 2^{k}.
\end{equation} for a suitable constant.  Of course, we can assume  $N \geq k$ from now on, otherwise there is nothing to prove. By first passing to an arbitrary finite subcollection, and then by proving bounds independent on the cardinality of such subcollection, we can reduce to the case of  $\Tt$ being finite.   We will need to define the further (finite) parameters
\begin{equation}
\label{ref3}
 2^M:=\max\left\{ 2^k , \sup_{J \in \cic{J}} \frac{1}{|J|}
\sum_{\mathsf{T}\in \Tt: I_\T\subset J}  |I_\mathsf{T}| \right\}, \qquad K=2^{\eps M}.
\end{equation}
  In particular, $2^M$ plays the role of BMO norm of the counting function in the next lemma. The proof (and in fact, the statement) is identical to that of \cite[eq.\ (6.19)]{DPTh2014} and is thus omitted.
\begin{lemma} There exists a decomposition $\Tt=\Tt' \cup \Tt''$  \begin{align}
\label{theLinftynorm}
&\Bigg\| \sum_{\mathsf{T}\in \Tt'} \cic{1}_{9K I_{\mathsf T}} \Bigg\|_\infty \leq 2^{9M},  \\ 
&  \sum_{\T \in \Tt''} |I_\T| \leq 2^{-k}. \label{thesmallpart}
\end{align}
\end{lemma}
In view of \eqref{thesmallpart}, $\Tt''$ can be removed from $\Tt$ for all practical purposes related to \eqref{thefinalbd}. By doing this, \eqref{theLinftynorm} now holds for $\Tt$ in place of $\Tt'$.
 At this point, we enucleate the Calder\'on-Zygmund decomposition procedure in the following lemma. 
 It is convenient to adopt the notations, for an interval $J \subset \R$,
 $$ \Tt_0(J):= \{\T \in \Tt_0: I_\T \subset J\}, \qquad  \Tt_1(J):= \{\T \in \Tt_1: I_\T \subset J\}
 \qquad \Tt(J)=\Tt_0(J) \cup \Tt_1(J).
 $$
 We think of the function $g_J$ appearing in the Lemma below as the projection of $f$ onto the time-frequency region corresponding to the generalized tents of $\Tt(J)$.
  \begin{lemma}[CZ decomposition] \label{lemmaCZ} 
Fix $J \in \cic{J}$. We can find   $g^{(J)}\in L^2(\R)$  with the properties 
\begin{equation} \label{lemmaCZL2} \begin{split}&\supp  g^{(J)} \subset 9KJ\\ &     \|g^{(J)}\|^2_2 \lesssim \eps^{-1} \min\left\{  2^{\left(1-\frac2q \right) M}|J|, 2^{3\eps M} 2^{\left(1-\frac2{p'} \right) N} \right\}  \end{split}
\end{equation}
and  
\begin{equation}
\begin{split} & 
\inf_{\T_n \in \Tt_0} |F_\phi(g^{(J)})(u_n,t_n,\eta_n)| > 2^{-\frac kq-1},
\\ &  \sup_{\T \in \Tt_1(J)}\sup_{(u,t,\eta) \in \T^*} \left|F_\phi(g^{(J)})(u,t,\eta)\right|\leq 2^{-\frac kq+1},
\\ &  2^{-\frac kq-5}   \leq \left(\frac{1}{|I_\T|} \int_{\T^*} |F_\phi(g^{(J)})|^2 \, \d u \d t \d \eta \right)^{\frac12} \leq 2^{-\frac kq+1} \qquad \forall \T \in \Tt_1(J).
 \label{lemmaCZsize} 
 \end{split} 
\end{equation}
\end{lemma}
With Lemma \ref{lemmaCZ} in hand, we are able to estimate \begin{equation}
\label{thefinalbd2}
2^M\leq C_{{\alpha},\phi,\overline{\eps},q}2^k.
\end{equation} Suppose $M>k$ (otherwise there is nothing to prove) and choose an interval $J$ such that the supremum in \eqref{ref3} is nearly attained. 
We have, using the estimates of \eqref{lemmaCZsize}, property \eqref{2sd1} of  strongly disjoint collections for $\lambda=2^{-\frac kq}$, and  the first bound of \eqref{lemmaCZL2} in the final estimate, that
\begin{align*} & \quad 2^{M-1} |J| \leq \sum_{ \mathsf{T}\in \Tt_0(J) } |I_{\mathsf T}|  +\sum_{ \mathsf{T}\in \Tt_1(J) } |I_{\mathsf T}|  \leq C_{{\alpha},\phi} 2^{\frac{2k}{q}}\|g^{(J)}\|_2^2 \leq C_{{\alpha},\phi} \eps^{-1} 2^{\frac{2k}{q}} 2^{\left(1-\frac2q\right) M} |J|
\end{align*}
and the bound  \eqref{thefinalbd2} follows by rearranging. Now, take $J$ large enough such that $\Tt(J)=\Tt$. Such a $J$ exists since $\Tt$ is finite.  By the same token, but this time using the estimate in \eqref{lemmaCZL2} which does not depend on $|J|$, and \eqref{thefinalbd2},  we obtain 
\begin{align*} & \quad 2^N= \sum_{ \mathsf{T}\in \Tt_0} |I_{\mathsf T}|  +\sum_{ \mathsf{T}\in \Tt_1 } |I_{\mathsf T}|  \leq C_{{\alpha},\phi}2^{\frac{2k}{q}} \|g^{(J)} \|_2^2 \leq C_{ {\alpha},\phi,\overline{\eps},q}2^{(\frac{2}{q}+3\eps) k}     2^{\left(1-\frac2{p'}\right) N} \leq C_{ {\alpha},\phi,\overline{\eps},q}2^{\frac{2}{p'}  k}     2^{\left(1-\frac2{p'}\right) N} ,
\end{align*}
and the estimate \eqref{thefinalbd} follows by rearranging. The proof of \eqref{cet} is complete up to the CZ Lemma \ref{lemmaCZ}.
 \section{Proof of the Calder\'on-Zygmund Lemma \ref{lemmaCZ}}  \label{sec8}
  Before we enter the argument for Lemma \ref{lemmaCZ}, we state the following projection Lemma, which will be used in the construction of the functions $g^{(J)}$. 
 \begin{lemma} \label{borerd}
Let $f \in L^p(\R),$ $1\leq p<2$. Let  $L\subset \R$ be an interval and  $$\lambda= 
|L|^{-\frac1p}\| f\cic{1}_L\|_p . $$
Let $\Xi\subset \R$ be a finite set and $Q\geq 0$ an integer. Then
$
f\cic{1}_L= g+ b
$
with $\supp g, \, \supp b \subset 3L$, and
\[
\begin{split}
&
\|g\|_2^2 \lesssim \lambda^2  (Q \#\Xi)^{1-\frac{2}{p'}}|L| ,  \\ &\int  b(x)  x^j \e^{-i \xi x } \, \d x =0 \qquad \forall j=0,\ldots, Q, \; \xi \in \Xi.
\end{split}
\]
 \end{lemma} 
 The proof of the lemma is a projection argument on the  finite dimensional subspace of $L^2(3L)$ given by
$$
\mathcal V_{\Xi, Q}= \mathrm{span} \{x^j \e^{-i \xi x }: j=0, \ldots, Q, \, \xi \in \Xi\}
$$
which is translation invariant when viewed as a subspace of $L^2(L)$. We give a proof of the more general Hilbert space version in the appendix, Lemma \ref{borerdH} therein.
\subsection{Construction of $g^{(J)}$}
  Recalling, from   \eqref{ref3}, the definition of $K$, set
  \[
\cic{L} := \textrm{{max.\ dyad.\ int.\ }} L\quad \textrm{{ s.t. }}\quad  
\|f\cic{1}_L\|_p > 2^{10}    K {|L|^{\frac 1p}}, \qquad \mathbb L = \bigcup_{L \in \cic{L}} L.
\]
 The intervals in $\cic{L}$ are pairwise disjoint.  We have that \begin{align}\nonumber& \|f\cic{1}_L\|_p \lesssim     K {|L|^{\frac 1p}},\\ \label{finov2} & \sup_{I\subset \R }\frac{1}{|I|}
\sum_{L \in \cic{L}: 3L \subset I} |L| \leq  1,\\ &
\label{distL}
E_{3K\cic{L}} = \bigcup_{L \in \cic{L}} \mathsf{T} (9KL) \times \R \subset E_{f,\lambda,p}.\end{align}
Note that \eqref{finov2} means that $\{3L: L\in \cic{L}\}$ have finite overlap  in the sense of \eqref{finov}.
 To obtain the third  property, observe that each $L\in \cic{L}$ is contained in some $I\in \cic{I}_{f,\lambda,p}$, and that $\lambda\leq 1$, see \eqref{A0}. It must actually be that $9KL$ is contained in $3I$, since
$$
\frac{|I|}{|L|} \geq \frac{2^{-3p} \|f\cic{1}_I\|_p^p}{2^{-10p}  K^{-p}\|f\cic{1}_L\|_p^p}\geq 2^{7} K.
$$
By the same token, referring to definition \eqref{outdyadint}, we see that
\begin{equation}
\label{theinters}
J \in \cic{J}, L \in \cic{L},  3L \cap 3KJ\neq \emptyset\implies 3L \subset 9KJ.
\end{equation}
Denote $\cic{L}(J)=\{L\in \cic{L}: 3L \cap  3KJ \neq \emptyset\}$. It will be convenient to observe that, using disjointness
\begin{equation} \label{LJ}
\sum_{L \in \cic{L}(J)} |L| \lesssim \min\left\{K|J|,
\left(2^{10}   K\right)^{-p} \|f\|_p^p\right\} \lesssim K\min\left\{|J|,1\right\} .
\end{equation}

From now on, $J \in \cic{J}$ is fixed.
We come to the definition  of the significant set of frequencies. Using the notation $\T=\T(I_\T, \xi_\T)$ for $\T\in \Tt$, we set
\begin{equation}
\label{freqdef} \Xi_L:= \left\{\xi: \exists \T \in \Tt(J) \textrm{ with } \xi=\xi_\T, 3L \subset 9KI_{\mathsf{T}}\right\}.
\end{equation}
A consequence of the bound \eqref{theLinftynorm} is that
\begin{equation}
\label{Linfcons}
\# \Xi_L \leq \inf_{x\in 3L}   \sum_{\mathsf{T}\in \Tt(J)} \cic{1}_{9K I_{\mathsf T}}(x)  \leq 2^{9M}.
\end{equation}
We can use the first inequality in \eqref{Linfcons} above, together with disjointness of $L \in \cic{L}(J)$ to obtain
\begin{equation}
\label{thecfbd}
\sum_{L\in \cic L(J)} |L|\#\Xi_L \leq \left\| \sum_{\mathsf{T}\in \Tt(J)} \cic{1}_{9K I_{\mathsf T}}\right\|_1    \lesssim K \min\left\{ 2^M|J|, 2^N\right\}. 
\end{equation}
See \eqref{ref3} and \eqref{cetlocpf1} for the definition of $M,N$.
For all $L \in \cic{L}(J)$, we  apply  Lemma \ref{borerd} to $f=f\cic{1}_L$, suitably rescaled, with the choice $\Xi=\Xi_L$ and $Q=\lceil 2^{5}\eps^{-1}\rceil$. This choice of $Q$ will be kept throughout the remainder of the proof.
The functions   $g_L$ and $b_L$ have the following properties:
\begin{align} & f\cic{1}_L = g_L +b_L, \qquad \supp g_L, \, \supp b_L \subset 3L \label{thesupp}
, \\ &  \label{gLbd} \|g_L\|_2^2   \lesssim \eps^{-1}   K^2 (\#\Xi_L)^{1-\frac{2}{p'}}|L|,  \\ \label{bLbd}&  \|b_L\|_1 \lesssim \eps^{-1}   K (\#\Xi_L)^{\frac12-\frac{1}{p'}}|L|\lesssim   \eps^{-1}  2^{6M}|L|, \\ \label{vanishing} &\int b_L(x)  x^j \e^{-i \xi x } \, \d x =0 \qquad \forall j=0,\ldots,Q,  \, \forall \xi \in \Xi_L.
     \end{align}
     The second inequality in \eqref{bLbd} is a consequence of \eqref{Linfcons}.
We finally set
\begin{equation}
\label{thedefofg}
g^{(J)}:= f\cic{1}_{3KJ\setminus \mathbb L}+ \sum_{L \in \cic{L}(J)} g_L, \qquad b:=   \sum_{L \in \cic{L}(J)} b_L.
\end{equation}
Note that, taking advantage of \eqref{bLbd} and of the finite overlap \eqref{finov2}, one may argue as   in Remark \ref{themftrick}  to obtain that 
$$
\left\{x \in \R:\mathrm{M} b(x)\gtrsim  \eps^{-1} 2^{7M}\right\} \subset \bigcup_{L \in \cic{L}} 3L \subset  \left\{x \in \R:\mathrm{M}_p f(x)\geq 2^{3} K\right\}.
$$
Hence, recalling the definition of $\cic{J}$, 
\begin{equation}
\label{techpf}
\sup_{I \in \cic{J}} \inf_{x \in I} M b(x)  \lesssim \eps^{-1} 2^{7M}.
\end{equation}
We now choose  $C_0$ to be the larger of the two   implicit absolute constants in \eqref{bLbd} and \eqref{techpf}. It can be easily tracked that $C_0 \leq 2^{12}$. 
Recalling \eqref{outexcset}, we will use this later with $I=I_\T$, $\T \in \Tt$.
\subsection{Proof of \eqref{lemmaCZL2}-\eqref{lemmaCZsize}}
\begin{proof}[Proof of \eqref{lemmaCZL2}]
Note that $\|f\cic{1}_{3KJ\setminus \mathbb L}\|_\infty \lesssim  K.$ Therefore
\begin{equation}
\label{outerpartL2}
\|f\cic{1}_{3KJ\setminus \mathbb L}\|_2^2 \lesssim \min\left\{K^2|J|,K^{2-p}  \|f\|_p^p\right\} \leq   2^{2\eps M}\min\left\{|J|,1\right\}.
\end{equation}
Relying on \eqref{gLbd} and on the observation that $\{3L: L\in \cic{L}_j\}$ have bounded overlap, \begin{align*}&\quad \Big\|\sum_{L \in \cic L(J)} g_L \Big\|_2^2
\lesssim  \sum_{L \in \cic L(J)} \|g_L \|_2^2 \lesssim  \eps^{-1} K^2 \left( \sum_{L\in \cic L(J)}  |L|\right)^{ \frac{2}{p'}} \left( \sum_{L\in \cic L(J)} |L|\#\Xi_L\right)^{1-\frac{2}{p'}}\\ & \lesssim \eps^{-1} K^3 \min\left\{ 2^{M(1-\frac{2}{p'})}|J|, 2^{N(1-\frac{2}{p'})}\right\}\lesssim  \eps^{-1}  \min\left\{ 2^{M(1-\frac{2}{q})}|J|, 2^{3\eps M} 2^{N(1-\frac{2}{p'})}\right\}\end{align*}
where the penultimate estimate follows from \eqref{LJ} and \eqref{thecfbd}, and the last by recalling the definition of $  \eps$ in \eqref{eps0} and of $K$ in \eqref{ref3}. The proof of \eqref{lemmaCZL2} is finished by combining the last display with \eqref{outerpartL2}.
\end{proof}
\begin{proof}[Proof of \eqref{lemmaCZsize}]
First of all, we note that 
$$
f=g^{(J)} + b + h, \qquad h:= f\cic{1}_{  (\mathbb L)^c\cap (3KJ)^c}+ \sum_{\substack{L  \in  \cic{L}\\ 3L \cap 3KJ =\emptyset}} f\cic{1}_L. 
$$
Also note that
\[
\begin{split}
& 
\quad |F_\phi(g^{(J)})| \geq |F_\phi(g^{(J)})|\cic{1}_{\mathcal{Z}\setminus E_{f,\lambda,p}} \geq \left(|F_\phi(f)| -  |F_\phi(b)|- |F_\phi(h)| \right)\cic{1}_{\mathcal{Z}\setminus E_{f,\lambda,p}} \\ & \geq F - \left( |F_\phi(b)|\cic{1}_{\mathcal{Z}\setminus E_{3K\cic{L}}} + |F_\phi(h)| \right)
\end{split}\]
where the inclusion \eqref{distL} was used in the last inequality. Therefore, by comparison with 
   \eqref{lemmasele1app}-\eqref{lemmasele3}, \eqref{lemmaCZsize} will follow if we prove 
\begin{align} \label{easyloctrick} &
\sup_{\mathsf{T}\in \Tt(J)} \mathsf{s} (F_{\phi}(h)) (\mathsf{T} )   \leq    2^{-10k}, \\ \label{chon} &\sup_{\mathsf{T}\in \Tt (J)} \mathsf{s}(F_{\phi}(b) \cic{1}_{\mathcal{Z}\setminus E_{3K\cic{L}}}) (\mathsf{T} ) )   \leq    2^{-10k}.
\end{align}
We have used that $\mathsf{s}\geq \max\{\mathsf s_2, \mathsf s_\infty\}$   and the obvious fact
$
  \mathsf s_2(H) (\T) \geq \left(|I_\T|^{-1} \int_{\mathsf T^*} |H|^2\right)^{1/2}
$.   

By virtue of the fact that $\supp h \cap 3KJ=\emptyset$,   we can apply the intermediate estimate \eqref{goodemb3b} to $h$ with the collection $\cic{J}$ replaced by $ \{J\}$, obtaining the chain of inequalities   \begin{equation}
\label{firstchoice}\begin{split}
&\quad\sup_{\mathsf{T}\in \Tt(J)} \mathsf{s} (F_{\phi}(h)) (\mathsf{T} ) \leq\|F_{\phi } (h)\cic{1}_{E_{ \{J\}}}\|_{L^{\infty}( \mathcal Z, \sigma_{\cic{\alpha}},\mathsf{s} )} \\ &\leq C_{{  \alpha},\phi,Q}   K^{-Q} \inf_{x\in J} \M h(x)   \leq C_{{  \alpha},\phi,Q}   K^{-Q} \inf_{x\in J} \M f(x)\\ & \leq C_{{  \alpha},\phi,Q}   K^{-Q} \lambda \leq  C_{{  \alpha},\phi,Q}  2^{-k \eps Q}  \leq C_{{  \alpha},\phi,Q}   2^{-32k} \leq 2^{-10k}  \end{split}
\end{equation}
 from which \eqref{easyloctrick} follows. We have used $h\leq f$ in the last step on the second line and \eqref{outdyadint} in the passage to the third line.
  The last step in \eqref{firstchoice} was obtained by    recalling from \eqref{ref3} that $k \leq  M$, that $Q=\lceil2^5\eps^{-1}\rceil$   and   that $k_0$ was chosen in order to have $
C_{{ \alpha},\phi,Q}  \leq 2^{k_0} \leq 2^{k}.
$

 We turn to \eqref{chon}. Fix a tent $\mathsf{T}=T(I,\xi)\in \Tt(J)$. Define 
$$\cic{L}(I)=\{L \in \cic{L}(J): L \cap 3K I \neq \emptyset\}, \qquad
b^{\textrm{in}}:= \sum_{ L \in \cic{L}(I) } b_L, \qquad b^{\textrm{out}}= b-b^{\textrm{in}}.
$$
The $b^{\textrm{out}}$ part, supported outside $3KI$, is also handled via \eqref{goodemb3b}. Using below that $\mathrm{M} b^{\textrm{out}}\leq \mathrm{M}b$ pointwise, and later \eqref{techpf},  we actually estimate the larger quantity \begin{equation} \label{bout}\begin{split}  
 &\quad \mathsf{s}(F_{\phi}(b^{\textrm{out}}) ) (\T) \leq  \|F_{\phi } (b^{\textrm{out}})\cic{1}_{E_{ \{I\}}}\|_{L^{\infty}( \mathcal Z, \sigma_{\cic{\alpha}},\mathsf{s} )}\\ & \leq  C_{{\cic \alpha},\phi,Q}   K^{-Q} \inf_{x \in I} \mathrm{M}b(x)  \leq C_0 \eps^{-1} C_{{\cic \alpha},\phi,Q}   2^{M(7-Q\eps)} \leq 2^{-10k}.  \end{split}
  \end{equation}
The last step has been obtained since $Q\geq 2^5\eps^{-1}$ and later using $M\geq k$ and the definition of $k_0$ \eqref{A0}.  Now we have to control the $b^{\textrm{in}}$ part. Observe that by our choice of $\Xi_L$ \eqref{freqdef}, and relying on \eqref{theinters},   $\xi \in \Xi_L $ for all $L\in \cic{L}(I)$.
Thus, taking also \eqref{thesupp}, \eqref{bLbd}  and \eqref{vanishing} into account, we have \[\begin{split} & \supp b_L \subset 3L, \qquad \|b_L\|_1 \leq C_0 \eps^{-1} 2^{6M}|L|, \\ &
\int b_L(x)  x^j \e^{-i \xi x } \, \d x =0 \qquad \forall j=0,\ldots,2^{5}\eps^{-1},
\end{split}
\] for all $L \in \cic{L}(I)$.
In other words, also in view of the finite overlap   \eqref{finov2}, comparing with Remark \ref{suffcond}, $b^{\mathrm{in}}$ satisfies the assumptions of  Lemma \ref{gloc3} with $\lambda= 3C_0 \eps^{-1} 2^{6M}$,  $\cic{L}$ replaced by  $\{3L:L \in \cic{L}(I)\}$, and $Q=2^5\eps^{-1}$.  Applying  the lemma, we obtain
$$\mathsf{s}\left(F_{\phi }(b^{\textrm{in}})\cic{1}_{\mathcal Z\setminus E_{3K\cic{L}}} \right) (\mathsf{T}) \leq  3C_0\eps^{-1}C_{{\cic \alpha},\phi,Q}  2^{M(6-Q\eps)}  \leq 2^{-10k}.
$$
 The last inequality is obtained in the same fashion as \eqref{bout}. Combining \eqref{bout} with the last display, we have finished the proof of \eqref{chon}, and, in turn, of \eqref{lemmaCZsize}.
\end{proof}

\setcounter{section}{0}
\setcounter{proposition}{0}
\renewcommand{\theproposition}{\thesection.\arabic{proposition}}
\setcounter{equation}{0}
\renewcommand{\theequation}{\thesection.\arabic{equation}}
\setcounter{figure}{0}
\setcounter{table}{0}

\appendix
\section{Remarks on the Hilbert space valued setting} 

We would like to remark that our main theorem can be easily generalized to the setting of Hilbert space valued functions. Let $\mathcal{H}$ be a separable Hilbert space with orthonormal basis $\{h_n: n \in \mathbb N\}$. By $L^{p}(\R; \mathcal{H})$ we denote the usual $\mathcal{H}$-valued Bochner spaces. The inner product associated to $\mathcal{H}$ is denoted as $\langle \cdot,\cdot\rangle_{\mathcal{H}}$. Then for Schwartz functions $f_1,f_2:\mathbb{R}\rightarrow\mathcal{H}$ and $f_3:\mathbb{R}\rightarrow \mathbb{C}$, one can naturally define
\begin{equation}
\mathsf{V}_{\vec \beta} (f_1,f_2,f_3)= \int_{\R\times (0,\infty)\times \R} \langle G_1,G_2\rangle_{\mathcal{H}} G_3 \, \d u \d t \d \eta, \quad G_j(u,t,\eta):=F_{\phi}(f_j)(u,t,\alpha_j \eta + \beta_j t^{-1}).
\end{equation}
where $\vec\beta$ is as before. We then have the following theorem, which is the Hilbert space valued analog of Theorem \ref{modelsumprop}.
\begin{theorem}\label{modelsumpropH}Let
\begin{equation}
\label{rangeH} \textstyle  1<p_1,p_2\leq \infty, \qquad \frac23 <r:=  \frac{p_1 p_2}{p_1+p_2}<\infty.
\end{equation}
 For all sets $A\subset \R$ of finite measure, and for all  Schwartz functions $f_1,f_2:\mathbb{R}\rightarrow\mathcal{H}$ we can find a subset $\widetilde A\subset A$ such that $|A| \leq 2 |\widetilde A|$ and 
\[
 \big|\mathsf{V}_{\vec \beta} (f_1,f_2,f_3)\big| \leq C_{\vec\beta, p_1,p_2, } \|f_1\|_{L^{p_1}(\R;\mathcal{H})} \|f_2\|_{L^{p_2}(\R;\mathcal{H})} |A|^{1-\frac{1}{r}} \qquad \forall |f_3| \leq \cic{1}_{\widetilde A}.
\]
\end{theorem}
We remark that Theorem \ref{modelsumpropH} implies the Hilbert space valued version of the full range of estimates for the bilinear Hilbert transform, which was not known before. The proof of Theorem \ref{modelsumpropH} utilizes the Carleson embedding theorem in the $\mathcal{H}$-valued setting, which can be easily generalized from the scalar valued case (Theorem \ref{bademb}) that has been presented, with the proper adaptation of sizes and outer $L^p$ spaces. Roughly speaking, the $\mathcal{H}$-adapted sizes are the same as in the scalar valued case but with the absolute value of $F$ replaced by the $\mathcal{H}$ norm. For example, in the case of generalized tents, let $F: \mathcal{Z}\rightarrow \mathcal{H}$ measurable, the sizes are defined as
\[
\begin{split}&
\mathsf{s}_2(F) (\T(I,\xi)):= \left(\frac{1}{|I|} \int_{\T^\ell_{\cic \alpha}(I,\xi)} \|F(u,t,\eta)\|_{\mathcal{H}}^2 \,  \d u \d t\d \eta  \right)^{\frac12},  \\ &\mathsf{s}_\infty(F) (\T(I,\xi)):=\sup_{(u,t,\eta) \in\T_{\cic \alpha}(I,\xi)} \|F(u,t,\eta)\|_{\mathcal{H}}, \\ &  \mathsf{s}:=\mathsf{s}_2+ \mathsf{s}_\infty.
\end{split}
\]

Similar basic properties hold for outer $L^p$ spaces in this setting, which we refer to \cite{DPOu} for details. In fact, the only part of the argument in the proof of Theorem \ref{modelsumpropH} that might need some further explanation is the adapted version of Lemma \ref{borerd}, which we state as the following.
\begin{lemma} \label{borerdH}
Let $f \in L^{p}(\R; \mathcal{H})$,  $1\leq p<2$. Let  $L\subset \R$ be an interval and  $$\lambda= 
|L|^{-\frac1p}\| f\cic{1}_L\|_{L^{p}(\R; \mathcal{H})} . $$
Let $\Xi\subset \R$ be a finite set and $Q\geq 0$ an integer. Then
$
f\cic{1}_L= g+ b
$
with $\supp g, \, \supp b \subset 3L$, and
\[
\begin{split}
&
\|g\|^2_{L^2(\R; \mathcal{H})} \lesssim \lambda^2  (Q \#\Xi)^{1-\frac{2}{p'}}|L| ,  \\ &\int  b(x)  x^j \e^{-i \xi x } \, \d x =0 \qquad \forall j=0,\ldots, Q, \; \xi \in \Xi.
\end{split}
\]
\end{lemma}
This projection lemma extends the result of \cite{BE, NOT} to the Hilbert space valued setting, as well as to a more general translation invariant finite dimensional subspace of $L^2(3L)$, which was unknown before even in the scalar-valued setting. We will need the following preliminary lemma in order to prove Lemma \ref{borerdH}.
\begin{lemma}\label{prelim}
Let $v_j:\R\to \mathbb C,$ $ j=1,\ldots,N$ be continuous functions  with the property that
$$v \in V:=\mathrm{span}_{\mathbb{C}} \{v_j:j =1,\ldots, N \}\implies v(\cdot+t)\in V \qquad \forall  t \in \R.$$  Then, for all intervals $I\subset \R$
$$
\sup_{x\in I}  |v(x)| \leq \sqrt{ \frac{N}{|I|}  \int_{3I}|v(x)|^2 \,\d x} \qquad \forall v \in V, \, \forall R>0.
$$
\end{lemma}
\begin{proof}
By translation and scaling invariance of the assumptions and conclusions, we can reduce to the case  $I=(-1/2,1/2)$. Consider $V_I= \mathrm{span}_{\mathbb{C}} \{v_j\cic{1}_{I},j=1,
\ldots,N\}$ as a linear subspace of $L^2(I)$ of dimension $n\leq N$. Let $\{w_1,\ldots,w_n\}$ be an orthonormal basis of $V_I$. 
Since $$
\int_{I} \sum_{j=1}^n |w_j(x)|^2 \, \d x =n
$$ 
there exists $x_0\in I $ such that
$$
|I| \sum_{j=1}^n |w_j(x_0)|^2 \leq n\leq N.
$$
Fix $v \in V_I$, and $x \in I$, and denote $v(\cdot)=\tilde v(\cdot +x_0-x)\cic{1}_I(\cdot) \in V_I$. We then have
\[
\begin{split}
&\quad |v(x)|^2 = |\tilde v(x_0)|^2 =\left| \sum_{j=1}^n \l \tilde v,w_j \r w(x_0)\right|^2 \leq   \left( \sum_{j=1}^n |w_j(x_0)|^2 \right) \left( \sum_{j=1}^n |\l \tilde v,w_j \r|^2  \right) \\ &  \leq  \frac{N}{|I|} \int_{I}|\tilde v(t)|^2 \,\d t \leq \frac{N}{|I|} \int_{3I}|  v(t)|^2 \,\d t
\end{split}
\]
which concludes the proof.
\end{proof}

\begin{proof}[Proof of Lemma \ref{borerdH}] There is no loss in generality with assuming that $\mathcal{H}$ is finite dimensional and $\{h_n\}$ is finite with cardinality $d$. By scaling invariance, we can reduce to the case $|I|=1$. The linear space
 $$
 V=\mathrm{span }_{\mathbb C} \{x\mapsto x^j \e^{-i\xi_\ell x}, j=0,\ldots, Q, \ell=1,\ldots, \#\Xi\}, \qquad N= (Q+1)\#\Xi.
 $$
 satisfies the assumptions of Lemma \ref{prelim} and let us use the same notation as in the lemma for $V_I$ and its orthonormal basis $\{w_1,
\ldots,w_N\}$. Let $$\mathcal V=\left\{h=\sum_{j=1}^d u_j h_j, \, u_j \in V, h_j \in \mathcal{H}\right\}.$$  
  We claim that
  \begin{equation}
\label{thecentral} \sup_{x \in I} \|h(x)\|_{\mathcal{H}}= \sup_{x \in I} \left(\sum_{j=1}^d |\l h(x),h_j \r| ^2\right)^{\frac12} \leq N^{\frac12} \|h\cic{1}_{3I}\|_{L^2(\R; \mathcal{H}) }
\end{equation}
which implies by interpolation that
  \begin{equation}
\label{thecentral2}  \|h\cic{1}_I \|_{L^{p'}(\R;\mathcal{H})}  \leq N^{\frac12-\frac{1}{p'}} \|h\cic{1}_{3I}\|_{L^2(\R; \mathcal{H}) }
\end{equation}
which is the one we will use. To prove \eqref{thecentral},
fix $x \in I$. Noting that $\l h(\cdot),h_j\r=u_j\in V$, we apply Lemma \ref{prelim} and obtain 
$$
|\l h(x),h_j \r|^2 \leq N \int_{3I} |\l h(x),h_j \r |^2 \d x 
$$
therefore
$$
\sum_{j=1}^d
|\l h(x),h_j \r|^2 \leq N \int_{3I} \sum_{j=1}^d  |\l h(x),h_j \r |^2 \d x = N  \int_{3I}  \| h(x)\|_{\mathcal{H}}^2\,  \d x =N \|h\cic{1}_{3I}\|_{L^2(\R; \mathcal{H}) }^2
$$
as claimed.
Using \eqref{thecentral2} and the assumption on $f$, $f\cic{1}_I$ defines a linear functional on the restriction of  $\mathcal V$  to $3I$ viewed as a subspace of $L^2(3I; \mathcal{H})$, with norm bounded by $\lambda N^{\frac12-\frac{1}{p'}}$. Then, using the Riesz representation theorem there exists
$g\in L^2(3I; \mathcal{H}) $ such that 
$$
\int_{I} \l f(x), h\r_{\mathcal{H}} \,\d x= \int_{3I} \l g(x), h\r_{\mathcal{H}} \,\d x \qquad \forall h \in \mathcal{V}.
$$
and satisfying the norm inequality.
Defining  the $\mathcal{H}$-valued function $b=f\cic{1}_I-g$ whose support is in  $3I$, we have 
$$
\int_{3I} \l b(x), h_j\r_{\mathcal{H}} x^k\e^{-i\xi_\ell x}  \,\d x= 0 \qquad \forall j, k, \ell,
$$
which implies that $b$ satisfies the vanishing moments conditions. The proof is complete. 
\end{proof}
\bibliography{biblioUMD}{}
\bibliographystyle{amsplain}

\end{document}